\documentclass[11pt,reqno]{amsart}
\usepackage{graphicx}
\usepackage{color}
\usepackage{amsmath,amssymb,amsthm,amsfonts,
mathrsfs,stmaryrd}
\usepackage{cases,indentfirst}
\usepackage{multicol}

\usepackage{amsmath, amsfonts, amsthm, amssymb,mathrsfs, amscd, graphicx, cite, color}

\makeatletter

\newtheorem{theorem}{Theorem}[section]
\newtheorem{lemma}{Lemma}[section]
\newtheorem{pro}{Proposition}[section]

\theoremstyle{definition}
\newtheorem{definition}{Definition}[section]

\theoremstyle{remark}
\newtheorem{remark}{Remark}[section]

\newcommand{\na}{\nabla}

\newcommand{\p}{\partial}
\newcommand{\n}{\rho}
\newcommand{\de}{\delta}
\newcommand{\si}{\sigma}

\newcommand{\om}{\Omega}
\newcommand{\la}{\label}

\newcommand{\ti}{\tilde}
\newcommand{\bn}{\begin{eqnarray}}
\newcommand{\en}{\end{eqnarray}}
\newcommand{\bnn}{\begin{eqnarray*}}
\newcommand{\enn}{\end{eqnarray*}}

\newcommand{\ba}{\begin{aligned}}
\newcommand{\ea}{\end{aligned}}
\newcommand{\be}{\begin{equation}}
\newcommand{\ee}{\end{equation}}

\def\p{\partial}
\def\norm[#1]#2{\|#2\|_{#1}}
\def\g{\gamma}

\def\ep{\varepsilon}
\def\a{\alpha}

\def\R{\mathbb{R}}
\def\r{\mathbb{R}^{3}}
\def\th{\theta}
\def\de{\delta}
\def\lap{\triangle}

\usepackage[numbers,sort&compress]{natbib}
\renewcommand{\div}{{\rm div}}

\theoremstyle{definition}
\numberwithin{equation}{section}

\allowdisplaybreaks

\begin{document}

%\author{Zhilei Liang$^{\dag}$}\thanks{$^{\dag}$ School of Economic  Mathematics, Southwestern University of Finance and Economics, Chengdu 611130,    China.  E-mail:\textbf{ zhilei0592@gmail.com}}
%\author{Dehua Wang$^{\ddag}$}\thanks{$^{\ddag}$ Department of Mathematics, University of Pittsburgh, Pittsburgh, 15260, USA.   E-mail: \textbf{dwang@math.pitt.edu}}
%

\author[Z.  Liang]{Zhilei Liang}
\address{School of Economic Mathematics, Southwestern  University of Finance and Economics, Chengdu  611130,  China}
\email{zhilei0592@gmail.com}
\author[D. Wang]{Dehua Wang}
\address{Department of Mathematics, University of Pittsburgh, Pittsburgh, PA 15260, USA}
\email{dwang@math.pitt.edu}

\title[Cahn-Hilliard-Navier-Stokes equations]  
{Stationary  Cahn-Hilliard-Navier-Stokes equations for the  diffuse interface model of  compressible  flows}
 
\begin{abstract}
A system of  partial differential equations for a diffusion interface model is considered for  the stationary motion of   two macroscopically immiscible, viscous   Newtonian fluids  in a three-dimensional bounded domain.  The governing equations consist  of  the stationary Navier-Stokes equations for compressible fluids  and  a stationary  Cahn-Hilliard type  equation   for  the mass  concentration difference.  Approximate solutions are constructed through a two-level approximation procedure, and the limit of the sequence of approximate solutions is obtained by a  weak convergence method.  New ideas and estimates are developed to establish the existence of   weak solutions with a wide  range of adiabatic exponent.
 \end{abstract}
	
\date{\today}
	
%{\textbf{AMS Subject Classification 2010:} 35Q35,  76N10, 35Q30,  34K21,  76T10.}

\maketitle

 \section{Introduction}

We are concerned with a diffuse interface model for   a mixture of two  viscous   fluids.
 The interface is usually caused by continuous but steep change of flow properties of immiscible or partially miscible fluids, which has been studied  largely in literature (see \cite{AM,low}).
An important analytical and numerical method to model such two-phase flows is the diffuse interface modeling (see \cite{yue}).
The hydrodynamical system of the mixture of two fluids is naturally the Navier-Stokes equations in each fluid domain with the kinematic and other conditions on the interface.
 On the other hand, the Allen-Cahn type  or  Cahn-Hilliard type of mixing models is commonly used based on the choice of the flux and production rate, see \cite{dre,liu} and the reference cited therein.
In this paper,  we  study  the  following stationary  Cahn-Hilliard-Navier-Stokes  equations of the  compressible  mixture of  fluid flows  in a  %three-dimensional
bounded domain $\om\subset\mathbb{R}^3$:
  \begin{equation}\label{1}
\left\{\ba
&\div(\n u)=0,\\
&\div(\n u\otimes u) = \div\left(\mathbb{S}_{ns}+\mathbb{S}_{c}-p\mathbb{I} \right)+\n g_{1}+g_{2},\\
&\div (\n u c)=\div (m\na \mu),\\
&\n\mu=\n \frac{\p f(\n,c)}{\p c}-\lap c,
\ea \right.
\end{equation}
where    $\n,\,u,\, c,\,\mu$ denote  the total density,   the mean velocity field,    the mass concentration  difference of the two components, and the chemical potential, respectively;     $m$ is the mobility that  is assumed    to be one for simplicity,  and $g_{1}$ and $g_{2}$  are given force terms.  We denote the Navier-Stokes    stress tensor  by \be\la{c0} \mathbb{S}_{ns} =\lambda_{1}\left(\na u+(\na u)^{\top}\right)+\lambda_{2} \div u\mathbb{I},\ee
where  $(\na u)^{\top}$ denotes the  transpose of $\na u$,    $\mathbb{I}$ is   the identity matrix,  and   $\lambda_{1},\,\lambda_{2}$ are   constants satisfying $\lambda_{1}>0, \, 2\lambda_{1}+3\lambda_{2}\ge0$.
%\bnn\lambda_{1}>0,\quad2\lambda_{1}+3\lambda_{2}\ge0.\enn
  In comparison with a   single fluid,   there is an additional   capillary  stress tensor  \be\la{b0}\mathbb{S}_{c}=-\na c\otimes \na c+\frac{1}{2}|\na c|^{2}\mathbb{I},\ee which describes  the  capillary effect related to the  surface energy.
  % penalizing mixing of the fluids and  variations of the concentration difference.
In this paper we assume the following form of  pressure    \be\la{b00} p=\n^{2}\frac{\p f(\n,c)}{\p \n} ,\ee  and the     free energy density
\be\ba\la{b1} f(\n,c)&=\n^{\g-1}+f_{mix}(\n,c)=\n^{\g-1}+H_{1}(c)\ln \n  +H_{2}(c),\ea\ee
with  the adiabatic exponent $\g>1$   and two given functions $H_{i}\,(i=1,2)$ of one variable. %will be specified  later.
We   remark that   the  mixed free energy density  $f_{mix}(\n,c)$   is
mainly  motivated   by  the   well-known logarithmic   form (cf. \cite{Fei,low}).
%, particularly,
%$ f_{mix}(\n,c)=\alpha_{1}\left(\frac{1-c}{2}\ln \n\frac{1-c}{2}\right)+\alpha_{2}\left( \frac{1+c}{2}\ln \n\frac{1+c}{2} \right)-\beta c^{2},$
%where the positive  constants  $\alpha_{1},\,\alpha_{2}, \beta$ are given.
We shall study the stationary equations \eqref{1} with
  the  following  boundary conditions:
\be\la{1a}u=0,\,\,\,\frac{\p c}{\p n}=0,\,\,\,\frac{\p \mu}{\p n}=0,\quad {\rm on} \,\,\,\p\om\ee
and the additional conditions:
\be\la{0r} \int \n(x)dx=m_{1}>0\quad {\rm and}\quad  \int \n(x) c(x)dx=m_{2},\ee
with two given   constants $m_1>0$ and $m_{2}$.

System \eqref{1} describes  the equilibrium state  for  the compressible mixture   of  two  macroscopically immiscible, viscous   Newtonian fluids  (cf. \cite{dre,low,tru}).   The goal  of this  paper is to investigate the existence of solutions to the problem \eqref{1}-\eqref{0r}, and give a rigorous  mathematical  justification of the existence of an equilibrium state for the mixture of fluids.  We recall  that the corresponding  evolutionary  system:
 \begin{equation}\la{1t}
\left\{\ba
&\p_{t}\n+\div(\n u)=0,\\
&\p_{t}(\n u)+\div(\n u\otimes u) = \div\left(\mathbb{S}_{ns}+\mathbb{S}_{c}-\n^{2}\frac{\p f(\n,c)}{\p\n}  \mathbb{I}\right),\\
&\p_{t}(\n c)+\div (\n u c)=\lap \mu,\\
&\n\mu=\n \frac{\p f(\n,c)}{\p c}-\lap c,
\ea \right.
\end{equation}
was derived  in Abels-Feireisl \cite[Section 2.2]{Fei},   and can be regarded as a  variant of the  model suggested by   Lowengrub-Truskinovsky \cite{low}.
%, so that     the compact  information  is available  even  in   vacuum zones.
For the evolutionary system \eqref{1t}, the existence of multi-dimensional renormalized weak solution of finite energy was obtained in \cite{Fei}
for  $\g>\frac{3}{2}$, and the one-dimensional weak and strong solutions were studied in \cite{ding, chen}.
%Additionally, there are also some   results  for compressible  time-evolutionary Cahn-Hilliard-Navier-Stokes equations  in   one dimension case.    Ding-Li \cite{ding} discussed      the  existence of global  classical solution, weak and strong solution. For both periodic and Dirichlet boundaries conditions, Chen-He-Mei-Shi \cite{chen} showed  that  the strong solutions   are  asymptotically stable  for  large initial disturbance.
 For the results on the stationary compressible Navier-Stokes equations, we refer the readers to  the books   \cite{lion,novo},  the papers \cite{novo1,freh,mpm,mp1}  and references therein.
For the models of the Cahn-Hilliard-Navier-Stokes  type  of  multi-component  viscous fluids   with phase transitions, in the case of  incompressible  fluids with matched  or non-matched  densities,  there  are several different approaches to  describe  the  evolutionary  diffusion   processes (cf. \cite{AGG11,bat,ERS,eb,hh,star}), and there is a large  literature on the existence and long-time behavior of solutions; see, e.g.,  \cite{ADG13,ables0,ables1,Gui,ables2,boyer2,boyer1,dss,fri,fri2,liu,star} and the references therein;  and we also mention  the existence  results in \cite{bis,ko1,ko2}  for   weak solutions  to   the  stationary  non-Newtonian  flows   in two and  three  dimensions.
As far as the    compressible flow  is concerned,    Lowengrub-Truskinovsky \cite{low}  developed  a thermodynamically and mechanically consistent model that  extends the Euler  and Navier-Stokes models to the case of compressible binary Cahn-Hilliard  mixtures; see also      \cite{an,AM,ds}  for other different approaches.
For general interface dynamics of the mixture of different fluids, solids or gas, including the Allen-Cahn type  and Cahn-Hilliard type, see \cite{an,AM,chan,Fei,ds,dss,bat,dre,eb,low,liu,plot,star,tru} and the references therein for the discussions and mathematical results.

% \section{Main results} \textbf{Notation:}
We now introduce the notation and state our main result. For two given matrices  $\mathbb{A}=(a_{ij})_{3\times3}$ and $\mathbb{B}=(b_{ij})_{3\times3}$,   we denote their   scalar product   by  $\mathbb{A}:\mathbb{B}=\sum_{i,j=1}^{3}a_{ij}b_{ij}$.  For two vectors $a,\,b\in \mathbb{R}^{3}$,  %then $ \mathbb{R}^{3\times3}\ni
$a\otimes b =(a_{i}b_{j})_{3\times3}.$  The characteristic function of a set $A$ is denoted by  $\textbf{1}_{A}$.   Let $C_{0}^{\infty}(\om,\,\r)$ be  the set of all smooth and
compactly supported functions $f: \om \mapsto \r$, and $C_{0}^{\infty}(\om)=C_{0}^{\infty}(\om,\,\mathbb{R})$. Similarly we denote by $C^{\infty}(\overline{\om})=C^{\infty}(\overline{\om},\,\mathbb{R})$ the set of uniformly smooth functions on   $\om.$
%Involving the Lebesgue  integral   in domain $\om$, we  use  simply $ \int f=\int_{\om}f(x)dx$.
We use $ \int f=\int_{\om}f(x)dx$ to denote the integral of $f$ on $\Omega$.
For  any $p\in [1,\infty]$  and integer $k\ge 0,$       $W^{k, p}(\om,\,\r)$ and $W^{k,p}(\om)$    are   the standard  Sobolev spaces (cf. \cite{ad})
%of all  $p$-integrable functions which are
valued in $\r$  or $\mathbb{R}$,   % up to their $k$-th order derivatives.
and  $ L^{p}=W^{0,p}  $ and $ H^{k}=W^{k,2}$.
 We denote by $\overline{f}$    the   weak limit of function $f.$

%Now we are ready to give the definition of weak solutions.

The definition of weak solutions is as follows.
\begin{definition}\la{defi}
The function $(\n,u,\mu,c)$ is    a weak solution to the problem  \eqref{1}-\eqref{0r}  if  for  some $p>\frac{6}{5}$ and  $\th>0$ with  $\th+\g>\frac{3}{2}$,
\bnn\ba &  \n\in L^{\g+\th}(\om ),\,\,\,\n\ge0\,\,a.e.\,\,{\rm in}\,\,\om, \\%\quad \eqref{0r} \,\,{\rm is\,\,satisfied;}\\
& u\in H_{0}^{1}(\om,\r),\quad \mu\in  H_{n}^{1}(\om), \quad c\in W_{n}^{2,p}(\om),\ea\enn
where $W_{n}^{k,p}(\om)=\{f\in W^{k,p}(\om): \,\,\frac{\p f} {\p n}|_{\p\om}=0\}$ for any positive integer $k$ and
$H_{n}^{1}=W_{n}^{1, 2}$, such that,

(i)  The system   $\eqref{1}$  is   satisfied in the distribution sense in $\Omega$,   i.e.,   for any $\Phi\in C_{0}^{\infty}(\om,\,\r)$,
\bnn \int \left(\n u\otimes u +\n^{2}\frac{\p f(\n,c)}{\p \n} \mathbb{I}-\mathbb{S}_{ns}-\mathbb{S}_{c} \right):\na \Phi =\int \left(\n g_{1}+g_{2}\right)\cdot\Phi,\enn
 and  for  any $\phi\in   C^{\infty}(\overline{\om})$,
\bnn \int \n u\cdot \na \phi =0,\quad
\int \n c u\cdot \na \phi =\int\na \mu\cdot \na \phi,\quad  \int \n \mu\phi -\n \frac{\p f(\n,c)}{\p c}\phi=\int\na c\cdot \na \phi;\enn
and \eqref{0r}  holds for some $m_{1}>0$ and $m_{2}\in \R.$

(ii) %After extended to $\r\backslash \om$   by zero,  the function  $(\n, u)$  satisfies     $\eqref{1}_{1}$   in  distributional  sense.  Moreover,   $(\n,u)$  is a renormalized solution by Di-Perna and Lion \cite{di}, i.e.,     \bnn \int_{\r}b(\n) u\cdot \na \phi dx=\int_{\r} \left(b'(\n)\n-b(\n)\right)\div u \phi dx,\enn
If $(\n,u)$ is prolonged by zero outside $\Omega$, then both the equation $\eqref{1}_{1}$ and
$$\div(b(\n) u)+ \left(b'(\n)\n-b(\n)\right)\div u=0$$ are satisfied in the distribution sense in $\r$,
where  $b\in C^{1}([0,\infty))$ with  $b'(z)=0$ if $z$ is  large enough.

 (iii) The   energy inequality  is valid
\bnn \int \left(\lambda_{1}|\na u|^{2}+(\lambda_{1}+\lambda_{2})(\div u)^{2}+|\na \mu|^{2}\right) dx\le \int \left(\n g_{1}+g_{2}\right)\cdot u.\enn
\end{definition}

%Our main result in this paper reads as follows:
We are ready to state our main result.
\begin{theorem}\la{t}
Assume that  $\om\subset\mathbb{R}^3$ is a bounded and simply connected domain with  $C^{2}$ smooth boundary,     \be\la{0}  g_{1},\,\, g_{2} \in  L^{\infty}(\om,\r),\ee
  the  functions   in \eqref{b1} satisfy
\begin{equation}\label{b2} |H_{i}(c)|+|H_{i}'(c)|\le \overline{H}<\infty,\quad \forall \, c\in\mathbb{R}, \quad i=1,2,\end{equation}
for some constant $\overline{H}$, % for each $c\in\mathbb{R}$,
and in addition,
\be\la{ss} \left\{\ba
&\g>\frac{5}{3}, %<\g\le \frac{5}{3},
\,\,\,\,{\rm if } \,\,\,\na \times g_{1}\equiv 0\;  \text{\rm  in } \Omega,\\
& \g>2, \,\,\,\, \text{ \rm otherwise}. %{\rm if } \,\,\,\na \times g_{1}\neq 0.
\ea\right.  \ee
Then, for any given constants $m_{1}>0$ and $m_{2}\in\R$,  the problem \eqref{1}-\eqref{0r}  admits a  weak solution  $(\n,u,\mu,c)$  in the sense of Definition \ref{defi}.
\end{theorem}

  We shall prove  Theorem \ref{t} via two  levels of approximations and weak convergence methods, which rely on  the heuristic approaches   in \cite{Fei,freh, jiang, jiang2, lion,novo,novo1, mpm,mp1}.    We remark that our stationary problem seems   worse than the time-evolutionary  one  because the energy inequality by itself gives less useful information  about the sequence of  approximate solutions, and it is much  more complicated than  the Navier-Stokes equations of the single fluid due to the coupling with the Cahn-Hilliard equations.
Our construction of weak solution in Theorem  \ref{t} of this paper follows the  spirit of  \cite{lion,novo,novo1,freh,mpm,mp1} for the stationary compressible Navier-Stokes equations, but we also need to overcome extra barriers from the coupled Cahn-Hilliard equations.
  The main difficulties and our  strategies  are  described  below.

 % \underline{{\it Construction of approximation \eqref{n1} with parameters $\ep$ and $\de$:}}

We first construct the approximate  system   \eqref{n1}, which is inspired by the   time-discretization  of equations \eqref{1t}. The main ideas for this approximation are the following:
%in Section 3  we begin  with  the  approximate  system   \eqref{n1}.
%
 (1) To guarantee the sufficient regularity on density $\n$,   we  add the  diffusion term $\ep^{4} \lap \n$ in the transport  equation and   an   artificial  pressure in the momentum equation.       Our choice of  $\ep^{4}$     as the   diffusion coefficient  makes it possible  to avoid the appearance of  new parameters and thus simplify   the  approximation  procedures.
  (2) In the   proof,   the total mass and  difference of  volume fraction should  be preserved,    namely,  both $\int_{\om}\n(x)dx$  and $\int_{\om} \n(x) c(x)dx$  are   constant.
This is    necessary from both the    physical and mathematical point of view,  and can be derived by  the    Hardy-Poincar\'e type inequality as well as  the well-posedness  of solutions.   For this purpose, we use  $\ep^{2} (\n -\n_{0})$ and $\ep (\n c-\n_{0}c_{0})$ in the approximation, which can be regarded as time discretization of
$\p_{t} \n$ and $\p_{t}(\n c)$ respectively.
  (3) {For fixed $\ep$ and $\de$,  we solve   \eqref{n1} by    the {Schaefer} fixed point theorem.
Some new ideas are needed in the proof.  
Firstly, the solution is not self-contained due to  the Neumann boundary conditions imposed  on $\mu$ and $c$.  To fix the constants, we add  compatible  integral conditions in the system  \eqref{n3}.  Secondly, we use the  conservative quantities \eqref{0r} and interpolation techniques  to obtain the  required estimates so that the uniform   {\it a priori} bounds can be closed. Next,  notice that the  pressure $p$  relies not only on  $\n$ but also on $c$,  and hence is  not   monotone   in $\n$ for all range of $c.$  In this connection,    we adopt  some idea in   \cite{Fei} and decompose}
 \be\la{ppp} p=\n^{2}\frac{\p f}{\p\n}=\n^{2}\frac{\p \ti{f}}{\p\n}-2\overline{H}\n\textbf{1}_{\{\n\le k\}}=\ti{p}-2\overline{H}\n\textbf{1}_{\{\n\le k\}},\ee
where    $\overline{H}\n\textbf{1}_{\{\n\le k\}}$ is  bounded for  some large but finite constant $k$.  See Remark \ref{r2.1} for the detail.
Finally, to avoid the appearance of higher order derivatives  of $c$,  we replace the capillary stress  % \bnn
 $\div\left( -\na c\otimes \na c+\frac{1}{2}|\na c|^{2}\mathbb{I}\right)
$ %\enn
in   \eqref{n1}  by  the equivalent expression %\bnn
$\left(\n\mu -\n \frac{\p f(\n,c)}{\p c}\right)\na c.$
%\enn

%\underline{{\it$\ep$-limit procedure for equations \eqref{n6}:}}

%In Section 4  we   build uniform  in $\ep$ the {\it a priori}  estimates to guarantee the  $\ep$-limit procedure  for  approximation sequences \eqref{n6},     by using   the compactness  theories   developed  in \cite{jiang,jiang2,lion}.

 Then we  shall establish  the    {\it a priori}  estimates uniform  in $\ep$ to guarantee the  $\ep$-limit procedure to obtain  the approximation sequence  \eqref{n6}     by using   the compactness  theories   developed  in \cite{jiang,jiang2,lion}.
 In the proof, we need   strong convergence of $\na c$  for  taking   limit   in  the momentum equation. For this purpose,   we shall make full use of  the properties obtained  from the higher order diffusion  in  the Cahn-Hilliard  equation; see for example the proof of    \eqref{b19}.
  Another difficulty   is  the non-monotonicity of  the pressure   with respect to $\n.$  Thanks to   the  decomposition technique  (see Remark \ref{r2.1})
  $H_{1}(c)$ is always   positive,   which leads to our desired estimates.

%\underline{{\it $\de$-limit in vanishing artificial pressure term}:}
Finally we need to show the $\de$-limit in the vanishing artificial pressure term. The proof shall be  based on  the compactness theories in \cite{jiang,jiang2}.
 The   difficulty     is that  the    approximation  sequence  does not provide any good estimate on the density  but  $\|\n\|_{L^{1}}$, which is different from the   evolutionary  equations for which the density  $\n$  is bounded in   $L^{\g}$ with  $\g>1$.  To overcome the  difficulty, we borrow some  ideas   {developed in  \cite{freh,mpm}  and  derive   the higher regularities    by means of  weighted pressure  estimates. However, we need to handle the  difficulties caused by  the Newmann boundary conditions and the appearance of  the strongly nonlinear stress tensor $\div\left( -\na c\otimes \na c+\frac{1}{2}|\na c|^{2}\mathbb{I}\right)$.  See  Lemma \ref{lem5.2} for the detail.}

 The rest of the paper is organized as follows.
 In Section 2, we construct the two-level approximation system, find the solution by the fixed point theorem, and  derive some energy estimates.
 In  Section 3, we derive the uniform estimates in $\ep$ and pass the limit as $\ep$ goes to zero.
 In Section 4, we derive the uniform estimates in $\de$ and pass the limit as $\de$ goes to zero to finally obtain the weak solution in Theorem \ref{t}.

 \bigskip

\section{Construction  of  approximation solutions}

We first set the following fixed constants:  %$\ep\in (0,1),\,\,\de\in (0,1)$.
 \be\la{w3}\ep\in (0,1),\,\,\de\in (0,1);\quad \n_{0} =\frac{m_{1}}{|\om|},\,\,c_{0}=\frac{m_{2}}{m_{1}},\ee  where  $m_{1},\,\,m_{2}$  are taken from  \eqref{0r}, and $|\om|$ denotes the Lebesgue measure of $\om$.
Then we consider  the following approximate system:
  \begin{equation}\label{n1}
\left\{\ba
& \ep^{2} \n + {\rm div}(\n u) =\ep^{4} \lap \n +\ep^{2} \n_{0},\\
&\ep^{2}\n u+{\rm div} (\n u\otimes u)+\na \left(\de \n^{4}+\n^{2} \frac{\p f(\n,c)}{\p\n}\right)+\ep^{4}\na \n\cdot \na u\\
&\qquad  = {\rm div} \mathbb{S}_{ns}+\n\mu\na c-\n \frac{\p f(\n,c)}{\p c}\na c+\n g_{1}+g_{2},\\
& \ep \n  c+\n u\cdot \na c=\lap \mu+\ep \n_{0}c_{0},\\
&\n\mu=\n \frac{\p f(\n,c)}{\p c}-\lap c, \ea\right.
\end{equation}
with the  boundary conditions
 \be\la{n1b} u=0,\quad \frac{\p \n} {\p n}=0,\,\,\, \frac{\p c} {\p n}=0,\,\,\frac{\p \mu} {\p n}=0,\quad {\rm on}\,\,\,\p\om.\ee
  \begin{remark}  A direct computation shows,  at   least formally,
  \bnn \n\mu\na c-\n \frac{\p f(\n,c)}{\p c}\na c=-\lap c\na c=\div\left( -\na c\otimes \na c+\frac{1}{2}|\na c|^{2}\mathbb{I}\right) =\div \mathbb{S}_{c}.\enn\end{remark}

  The   following lemma is concerned with the solvability of $\eqref{n1}_{1}$, and its proof can be found  in \cite[Prop. 4.29]{novo}.
  \begin{lemma}[\cite{novo}, Proposition 4.29]
  \la{lem2.1}
Suppose   \be\la{a5} v\in  W_{0}^{1,\infty}(\om,\r):=\{v\in W^{1,\infty}(\om,\r),\,\,\,v|_{\p\om}=0\}.\ee
Then  there exists  a  function  $\n=\n(v)\in W^{2,p}(\om)\,\,(1<p<\infty)$  such that     for any  $\eta\in C^{\infty}(\overline{\om})$,
  \be\la{a3}  \ep^{4}\int \na \n\cdot \na \eta -\int \n v\cdot\na \eta +\ep^{2}\int (\n-\n_{0})\eta=0, \ee
  where   $\ep>0$ is a fixed constant.
Moreover,
\be\la{a17} \n\ge0\,\,a.e.\,\,{\rm in}\,\,\om,\quad\|\n\|_{L^{1}}=m_{1},\quad  \|\n\|_{W^{2,p}}\le C(\ep,p, \|v\|_{W^{1,\infty}}).\ee
\end{lemma}

Next,  we consider the    Neumann boundary  problem
\be\la{n}\lap \n= \div b \quad {\rm with} \quad \frac{\p \n} {\p n}\Big|_{\p \om}=0.\ee
\begin{lemma}[\cite{novo}, Lemma 4.27]  %\cite[Lemma 4.27 ]{novo}
\la{lem2.3}
 Let $p\in (1,\infty)$  and  $b\in  L^{p}(\om,\r)$ be given.  Then  the problem  \eqref{n} admits  a  solution  $\n\in W^{1,p}(\om)$,   satisfying
\bnn \int \na \n\cdot \na \phi=\int b\cdot \na \phi,\quad \forall\,\,\,\phi\in C^{\infty}(\om),\enn
and  the  estimates
\bnn\|\na \n\|_{L^{p}}\le C(p,\om)\|b\|_{L^{p}}\quad{\rm and }\quad
\|\na \n\|_{W^{1,p}}\le C(p,\om)(\|b\|_{L^{p}}+ \|\div b\|_{L^{p}} ).\enn
\end{lemma}

Our main task in this section is to prove the following theorem.

\begin{theorem} \la{t3.1} Under the conditions \eqref{0}, \eqref{b2} and   \eqref{w3},
  for any fixed $\ep>0$   the problem
\eqref{n1}-\eqref{n1b}  has  a solution  $(\n_{\ep},u_{\ep},\mu_{\ep},c_{\ep})$, such that for all $p\in (1,\infty),$
 \be\la{a19}  0\le \n_{\ep}\in W^{2,p}(\om),\quad \|\n_{\ep}\|_{L^{1}(\om)}=m_{1},\ee
 \be\la{a20} u_{\ep}\in  W_{0}^{1,p}(\om,\r)\cap W^{2,p}(\om,\r),\quad (\mu_{\ep},\,c_{\ep})\in W^{2,p}(\om)\times W^{2,p}(\om).\quad \ee
\end{theorem}

 \begin{proof}
We will  prove  Theorem \ref{t3.1} by    the fixed point theorem.  Setting
\be\la{a11}  (v,\ti{\mu},\ti{c})\in \mathcal{W}:=W_{0}^{1,\infty}(\om,\r)\times W_{n}^{1,p}(\om)\times W_{n}^{1,p}(\om),\ee
where  $W_{n}^{1,p}(\om)=\{f\in W^{1,p}(\om):\,\,\ \,\frac{\p f} {\p n}|_{\p\om}=0\}$  and    $W_{0}^{1,\infty}$ is  from  \eqref{a5}.
Let us consider   the   elliptic system    of  $(u,\mu,c)$:
 \begin{equation}\label{n3}
\left\{\ba
&{\rm div} \mathbb{S}_{ns}=F^{1}(v,\ti{\mu},\ti{c})\\&
\quad\quad:=\ep^{2} \n v+{\rm div} (\n v\otimes v)+\na (\de\n^{4}+\n^{2}\frac{\p f(\n,\ti{c})}{\p\n})+\ep^{4}\na \n\cdot \na v \\
&\qquad\qquad\qquad\qquad\qquad\qquad+\n \frac{\p f(\n, \ti{c})}{\p \ti{c}}\na \ti{c}-\n\ti{\mu}\na \ti{c}-\n g_{1}-g_{2},\\
&\lap \mu= F^{2}(v,\ti{\mu},\ti{c}):=\ep \n\ti{c}+\n v\cdot \na \ti{c}-\ep \n_{0}c_{0},\\
&\lap c=F^{3}(v,\ti{\mu},\ti{c}):=\n  \frac{\p f(\n, \ti{c})}{\p \ti{c}}-\n\ti{\mu},\\
&  \int  \n\tilde{c}=m_{2}+\ep\int(\n_{0}-\n)\tilde{c}-\ep^{3} \int \na\n \cdot\na \tilde{c},\quad \int \n \tilde{\mu}=\int \n \frac{\p f(\n,\tilde{c})}{\p \tilde{c}},\\
&{ \int  \n c=m_{2}+\ep\int(\n_{0}-\n)c-\ep^{3}\int \na\n \cdot \na c,\quad \int \n \mu=\int \n \frac{\p f(\n,c)}{\p c},}\\
& u=0,\,\,\,\frac{\p \mu}{\p n}=0,\,\,\,\, \frac{\p c}{\p n}=0,\quad {\rm on}\,\,\,\p\om,\ea \right.
\end{equation}    where %the constants $C_{1}(\ep),\,C_{2}(\ep)$ are given, function
$\n=\n(v)$ is determined   in   Lemma \ref{lem2.1}.
For   any given  $(v,\ti{\mu},\ti{c})$ satisfying $\eqref{n3}_4$,   the system \eqref{n3} has  a solution \be\la{nn1} (u,\mu,c):= A[(v,\ti{\mu},\ti{c})].\ee
Applying  the   $L^{p}$ regularity estimates  (cf.\cite{gt}),  we have   $$\|(u,\mu,c)\|_{W^{2,p}} \le C \|(F^{1},F^{2},F^{3})\|_{L^{p}}<\infty.$$

\begin{remark}
The  condition $\eqref{n3}_4$ guarantees  $\int F^{2}=\int F^{3}=0$  which is compatible with the  Neumann  boundary conditions in \eqref{n3}.
In fact,  by this condition together with \eqref{w3} and Lemma \ref{lem2.1},  one has
\bnn\ba \int F^{2} &=\int \left(\ep \n\ti{c}+\n v\cdot \na \ti{c}-\ep \n_{0}c_{0}\right)
 =\int\left(\ep \n\ti{c}+\n v\cdot \na \ti{c}  \right)-\ep m_{2}\\
&=\ep\int  \n\ti{c}+\ep^{4}\int\na \n \cdot \na  \ti{c}+\ep^{2}\int (\n-\n_{0})\tilde{c}-\ep m_{2}=0.\ea\enn
The second equality of the condition $\eqref{n3}_4$ yields $\int F^{3}=0$ immediately.
{Finally, we  note that the condition $\eqref{n3}_{5}$ is for the uniqueness of $\mu$ and $c.$ The two conditions $\eqref{n3}_{4}$ and $\eqref{n3}_{5}$ coincide after the fixed point argument.}
 \end{remark}

\begin{pro}\la{c} Suppose that $(u,\mu,c)$ is  a solution to \eqref{n3} and the operator $A\,:\mathcal{W}\mapsto \mathcal{W}$ is defined  in \eqref{nn1}.
   Then,  the  set of   possible  fixed points 
\be\la{a4} \left\{(u,\mu,c)\in \mathcal{W}   \left|\ba&(u,\mu,c):= \sigma A[(u,\mu,c)]\\
&   {\rm for\,\,some}\,\,\,\sigma \in (0,1]\,\,{\rm and}\quad  \n=\n(u)
\ea\right.
\right\}\ee  is bounded, where $\mathcal{W}$ is defined in \eqref{a11}.
\end{pro}

A  standard argument shows  that    $A$   is  compact and   continuous in $\mathcal{W}$.  Therefore,   using Proposition  \ref{c},   we conclude from the  {Schaefer}  Fixed Point  Theorem (Chap. 9, Th. 4 in  \cite{evans1})     that  $ (u,\mu,c):= A[(u,\mu,c)]$ with $\n=\n(u)$.
This  and   Lemma \ref{lem2.1}  guarantee the existence of solution   $(\n_{\ep},u_{\ep},\mu_{\ep},c_{\ep})$  to    \eqref{n1}-\eqref{n1b}.  Consequently,        \eqref{a19}  follows directly from \eqref{a17}.

\smallskip
It remains to prove  Proposition \ref{c} as well as \eqref{a20}.
\smallskip

{\bf Proof of Proposition \ref{c}.}
It suffices to   show that there is  a   constant   $M<\infty$  independent of    $\si$ such that
  \be\la{a8} \|(u,\mu,c)\|_{\mathcal{W}} <M,\ee
where  $(\n,u,\mu,c)$ solves  \begin{equation}\label{n4}
\left\{\ba
&  \ep^{2} \n + {\rm div}(\n u) =\ep^{4} \lap \n +\ep^{2}\n_{0},\\
&{\rm div} \mathbb{S}=\sigma F^{1}(u,\mu,c),\\
&\lap \mu=\sigma F^{2}(u,\mu,c), \\
&\lap c=\sigma F^{3}(u,\mu,c),\\
& \int  \n c=m_{2}+\ep\int(\n_{0}-\n)c-\ep^{3}\int \na\n \cdot \na c,\quad \int \n \mu=\int \n \frac{\p f(\n,c)}{\p c},\\
&u=0,\quad \frac{\p \n}{\p n}=0,\,\,\,\frac{\p \mu}{\p n}=0,\,\,\,\frac{\p c}{\p n}=0,\quad {\rm on}\,\,\,\p\om.\ea \right.
\end{equation}
We divide the proof   into several steps.

 {{\it Step 1.}}
It follows directly from   $\eqref{n4}_{1}$  that  $\|\n\|_{L^{1}}=m_{1}.$
Multiplying  $\eqref{n4}_{1}$ by $\frac{1}{2}|u|^{2}$ and  $\eqref{n4}_{2}$ by $u$  respectively,  we get
  \be\la{1.1}\ba &\frac{\ep^{2}\si}{2}\int  (\n+\n_{0})|u|^{2}  + \si\int  u\cdot \na \left(\de\n^{4}+\n^{2}\frac{\p f}{\p \n}\right) \\
  &\quad +\int \mathbb{S}_{ns}:\na u+\si\int  \n \frac{\p f}{\p c} (u\cdot\na) c-\si\int \n\mu  (u\cdot\na) c =\si\int \left(\n g_{1}+g_{2}\right) \cdot u.\ea\ee
Using  $\eqref{n4}_{1},$  one has
  \bnn\ba &\int  u\cdot\na \left(\de\n^{4}+\n^{2}\frac{\p f}{\p \n}\right)\\
&=\int \n u\cdot\na \left( \frac{4\de}{3}\n^{3}+\frac{\p (\n f)}{\p \n}\right)+\int u\cdot \left(\na \n\frac{\p (\n f)}{\p \n} -\na (\n f)\right) \\
   &=-\int \div (\n u) \left( \frac{4\de}{3}\n^{3}+\frac{\p (\n f)}{\p \n}\right) -\int \n \frac{\p f}{\p c}(u\cdot \na) c\\
  &=\ep^{2}\int \left(\frac{4\de}{3}\n^{3}+\frac{\p (\n f)}{\p \n}\right)(\n-\n_{0}) +\ep^{4} \int \left(4\de\n^{2}+\frac{\p^{2} (\n f)}{\p \n^{2}}\right)|\na \n|^{2}\\
  &\quad+\ep^{4} \int  \frac{\p^{2} (\n f)}{\p \n\p c}\na \n\cdot \na c-\int \n \frac{\p f}{\p c}(u\cdot \na) c. \ea\enn
Substitute the above  into  \eqref{1.1}   to obtain
   \be\la{1.4}\ba & \frac{\ep^{2}\si}{2}\int (\n+\n_{0})|u|^{2}
 +\ep^{2}\si\int\left(\frac{4\de}{3}\n^{3}+ \frac{\p(\n f)}{\p\n}\right)(\n-\n_{0}) \\
 &\quad +\int\mathbb{S}_{ns}:\na u-\si\int \n\mu  (u\cdot\na) c +\ep^{4}\si \int\left(4\de\n^{2}+\frac{\p^{2} (\n f)}{\p \n^{2}}\right)|\na \n|^{2}\\
  &=\si\int  \left(\n g_{1}+g_{2}\right)\cdot u-\ep^{4} \si\int \frac{\p^{2}(\n f)}{\p\n \p c}\na \n\cdot \na c .\ea\ee
 Next,  multiplying   $\eqref{n4}_{3}$ by $\mu$ and $\eqref{n4}_{4}$ by $c$  gives rise to
 \be\la{1.3}\ba &\int  |\na \mu|^{2}+ \si\int \n\mu (u\cdot\na) c+\ep\int  |\na c|^{2}=\ep\si\int  \n_{0}c_{0}\mu-\ep \si\int \n\frac{\p f}{\p c}c.\ea\ee
Combining   \eqref{1.4}  with  \eqref{1.3}  leads to
  \be\la{1.41}\ba & \frac{\ep^{2}\si}{2}\int  (\n+\n_{0})|u|^{2}
+\ep^{2} \si\int \left(\frac{4\de}{3}\n^{3}+ \frac{\p (\n f)}{\p \n}\right)(\n-\n_{0}) \\
 &\quad  +\ep\int |\na c|^{2}+\int  |\na \mu|^{2}  +\int \mathbb{S}_{ns}:\na u +\ep^{4}\si \int  \left(4\de\n^{2}+\frac{\p^{2} (\n f)}{\p \n^{2}}\right)|\na \n|^{2} \\
  &=\si\int \left(\n g_{1}+g_{2}\right)\cdot u-\ep^{4} \si\int \frac{\p^{2} (\n f)}{\p \n\p c}\na \n\cdot \na c +\ep\si\int \n_{0}c_{0}\mu  -\ep\si\int \n \frac{\p f}{\p c} c.\ea\ee
{We first assume that} \be\la{ee1} {H_{1}(c)\ge1\quad {\rm for\,\,all}\quad c\in \mathbb{R}.}\ee (See Remark \ref{r2.1} for the opposite case). Then, from \eqref{b1}  we compute
\be\la{pp} 4\de\n^{2}+\frac{\p^{2} (\n f)}{\p \n^{2}}\ge 4\de\n^{2}+ \g(\g-1)\n^{\g-2}+ \n^{-1}\ge0.\ee
Therefore, \bnn \ba&\int  \left(4\de\n^{2}+\frac{\p^{2} (\n f)}{\p \n^{2}}\right)|\na \n|^{2}\ge  \int \left( \de|\na \n^{2}|^{2}+4(\g-1)\g^{-1}|\na \n^{\frac{\g}{2}}|^{2}+4 |\na \sqrt{\n}|^{2}\right)\ea\enn
and
\bnn\ba\int \left(\frac{4\de}{3}\n^{3}+\frac{\p (\n f)}{\p \n}\right)(\n-\n_{0})\ge \int \left(\frac{\de}{3}\n^{4}+ \n f(\n,c)\right)-\int \left(\frac{\de}{3}\n_{0}^{4}+\n_{0} f(\n_{0},c)\right).\ea\enn
Taking the   last two inequalities into accounts, we estimate
 \eqref{1.41}  as
  \be\la{1.8}\ba & \ep^{2}\int (\n+\n_{0})|u|^{2}
 +\ep^{2} \si\int \frac{\de}{3}\n^{4} +\ep^{2}\si\int \n  f(\n,c)\\
 &\quad+\ep\int |\na c|^{2}  +\int |\na \mu|^{2}  +\int \mathbb{S}:\na u\\
 &\quad+\ep^{4}\si \int \left(\de \left|\na \n^{2}\right|^{2}+ 4(\g-1)\g^{-1}|\na \n^{\frac{\g}{2}}|^{2}+ 4 |\na \sqrt{\n}|^{2}\right) \\
  &\le \si\int \left(\n g_{1}+g_{2}\right) \cdot u+\ep^{2} \si\int \frac{\de}{3}\n_{0}^{4} -\ep^{2}\si\int \n_{0} f(\n_{0},c)\\
  &\quad-\ep^{4} \si\int  \frac{\p^{2} (\n f)}{\p \n\p c}\na \n\cdot \na c + \ep\si\int \n_{0} c_{0}\mu -\ep\si\int \n \frac{\p f}{\p c} c.\ea\ee
\begin{remark} \la{r2.1} {If \eqref{ee1} fails,  we  follow the idea in \cite{Fei} and  express the  $f(\n,c)$ in \eqref{b1} as} \bnn  f(\n,c) = \underbrace{\n^{\g-1}+ \left(H_{1}(c)+2\overline{H} \textbf{1}_{\{\n\le k\}}\right)\ln \n+H_{2}(c)}_{\ti{f}(\n,c)}-2\ln \n\overline{H}\textbf{1}_{\{\n\le k\}}, \enn
where $\overline{H}$ is taken from \eqref{b2}, and the constant  $k$   is large but fixed.
Let us   decompose   the pressure function    as
 \be\la{ext} p=\n^{2}\frac{\p f}{\p\n}=\n^{2}\frac{\p \ti{f}}{\p\n}-2\overline{H}\n\textbf{1}_{\{\n\le k\}}=\ti{p}-2\overline{H}\n\textbf{1}_{\{\n\le k\}}\ee
and replace $p=\n^{2}\frac{\p f}{\p\n}$ with $\ti{p}=\n^{2}\frac{\p \ti{f}}{\p\n}$  in \eqref{1.1}.
 We claim that  \eqref{pp} is  also  valid.  To see this,   if  $\n\le k,$ we have
  \bnn\ba  4\de\n^{2}+\frac{\p^{2} (\n \ti{f})}{\p \n^{2}}&= 4\de\n^{2}+ \g(\g-1)\n^{\g-2}+(H_{1}+2\overline{H})\n^{-1}\\
  &\ge4\de\n^{2}+ \g(\g-1)\n^{\g-2}+ \overline{H} \n^{-1}>0,\ea\enn
owing to $ H_{1}(c) +2\overline{H}>\overline{H}$;   while  if   $\n>k$,
\bnn \ba 4\de\n^{2}+\frac{\p^{2} (\n \ti{f})}{\p \n^{2}}
&=4\de\n^{2}+\frac{\p^{2} (\n f)}{\p \n^{2}}\\
&\ge 4\de\n^{2}+ \g(\g-1)\n^{\g-2}+ H_{1}(c) \n^{-1}\ge 4\de\n^{2} -\overline{H}\n^{-1}>2\de\n^{2}>0,\ea\enn
as long as $k=k(\de, \overline{H})$ is taken to be large enough.
However,  the following   extra  term  will be  induced  by    the decomposition  \eqref{ext},
\bnn \int  2\overline{H}\n\textbf{1}_{\{\n\le k\}}\div u.\enn
Fortunately,  it can be bounded by $\|\na u\|_{L^{2}}$ because $\n \textbf{1}_{\{\n\le k\}}$  is bounded.
Without loss of generality,  in what follows,  we always assume that, for all $c\in \mathbb{R},$ $H_{1}(c)$ is positive and bounded from below.
\end{remark}

 {{\it Step 2.}}
Let us deal with    the terms   on the right-hand side of \eqref{1.8}.
Thanks to  \eqref{b1},  \eqref{b2}, \eqref{0}, \eqref{w3},  and  the  H\"{o}lder   inequality,    the first three terms   satisfy
\be\la{3.31a}\ba& \si\int \left(\n g_{1}+g_{2}\right)\cdot u+\ep^{2}\si\int  \frac{\de}{3}\n_{0}^{4}+\ep^{2}\si\int \n_{0}f(\n_{0},c)\\
&\le C\|u\|_{L^{6}}  \|\n\|_{L^{\frac{6}{5}}}\|g_{1}\|_{L^{\infty}}+C\|u\|_{L^{6}}\|g_{2}\|_{L^{\infty}}+C  \\
&\le  C\left(1+\|\n\|_{L^{\frac{6}{5}}}^{2}\right)+\frac{\lambda_{1}}{2}\|\na u\|_{L^{2}}^{2}.\ea\ee
Throughout  this section,  the positive constants $C,C_{i}\,(i=1,2,\cdot\cdot\cdot)$   may depend  on $g_{1}$, $g_{2}$, $ \lambda_{1}$, $ m_{1}$, $m_{2}$,
$\de$, $\g$,  $\overline{H}$, $|\om|$,  but not on $\ep$ or $\si$.

Using \eqref{b1} and   \eqref{b2} again, one has
\be\la{3.32a}\ba \ep^{4}\si\int \frac{\p^{2} (\n f)}{\p \n\p c}\na \n\cdot \na c&\le C\ep^{4}\si\|\na c\|_{L^{2}} \| (1+\ln\n)\na \n\|_{L^{2}}\\
&\le  \frac{\ep}{4}\|\na c\|_{L^{2}}^{2} + C_{1}\ep^{7}\si\int \left(\de\left|\na \n^{2}\right|^{2}  + 4 |\na \sqrt{\n}|^{2}\right).\ea\ee
It follows from  $\eqref{n4}_{4}$  that
 \be\la{bb9}  \int  \n \mu= \int  \left(\n \frac{\p f}{\p c}+\si^{-1}  \lap c\right)=\int \n \frac{\p f}{\p c}  \le C \left( \|\n\ln\n\|_{L^{1}} +1\right).\ee
 Then we have, from \eqref{bb9} together with  \eqref{w3} and the Poincar\'e  inequality,
\be\ba\la{bb12}
\int \mu%=&\frac{1}{\n_{0}|\om|}\int m_{1}\mu
&=\frac{1}{\n_{0}}\int \n\left(\frac{1}{|\om|}\int \mu\right)\\
&=\frac{1}{\n_{0}}\int \n \mu-\frac{1}{\n_{0}}\int \n\left(\mu-\frac{1}{|\om|}\int \mu\right)\\
&\le C \left( \|\n\ln\n\|_{L^{1}} +1\right)+C \|\n\|_{L^{\frac{6}{5}}}\|\na \mu\|_{L^{2}},\ea\ee
which implies  \bnn\ba
 \|\mu\|_{L^{1}}
&\le C \|\na \mu \|_{L^{2}} + C \left( \|\n\ln\n\|_{L^{1}} +1\right)+C \|\n\|_{L^{\frac{6}{5}}}\|\na \mu\|_{L^{2}}\\
&\le C(1+\|\na \mu \|_{L^{2}})(1 + \|\n\|_{L^{\frac{6}{5}}}),\ea\enn where we have  used   $\|\n\ln\n\|_{L^{1}}\le C+\|\n\|_{L^{\frac{6}{5}}}$, owing  to the interpolation and  $\|\n\|_{L^{1}}=m_{1}.$
Thus,   \be\la{3.30}\ba
 \|\mu\|_{L^{p}}\le C\left(1+\|\na \mu \|_{L^{2}})(1+ \|\n\|_{L^{\frac{6}{5}}}\right),\quad \forall\,\,\,p\in [1,6].\ea\ee
 Thanks to \eqref{w3} and   $\eqref{n4}$,  one has
\be\la{w7} \ba  \int  \n c &= m_{2}+\ep\int  (\n_{0}-\n) c+\ep^{3}\int c\lap \n  \\
 &= m_{2}+\ep\int  (\n_{0}-\n) c+\ep^{3}\int   \n\lap  c\\
&\le C +C\ep\|\n\|_{L^{\frac{6}{5}}}\|\na c\|_{L^{2}} +C\ep^{3} ( \|\n\|_{L^{\frac{12}{5}}}^{2}\|\mu\|_{L^{6}}+\|\n^{2}\ln\n\|_{L^{1}}),\ea\ee
where we have used the estimate: %the inequality sign owes to
\bnn\ba \int (\n-\n_{0})c=\int \n c-\frac{1}{|\om|}\int m_{1}c=\int \n\left(c-\frac{1}{|\om|}\int c\right)\le C\|\n\|_{L^{\frac{6}{5}}}\|\na c\|_{L^{2}}.\ea\enn
With the aid of \eqref{w7},  the same method as   \eqref{bb12} yields, \be\ba\la{bb13}
\int c&=\frac{1}{\n_{0}}\int \n c -\frac{1}{\n_{0}}\int \n\left(c-\frac{1}{|\om|}\int c\right)\\
&\le C  +C \|\n\|_{L^{\frac{6}{5}}}\|\na c\|_{L^{2}}+C\ep^{3} ( \|\n\|_{L^{\frac{12}{5}}}^{2}\|\mu\|_{L^{6}}+\|\n^{2}\ln\n\|_{L^{1}}).\ea\ee
Thus, for  $p\in [1,6],$
\be\la{3.30c}\ba
 \|c\|_{L^{p}}
&\le   C (1+\|\n\|_{L^{\frac{6}{5}}})(1+\|\na c\|_{L^{2}})+C\ep^{3} ( \|\n\|_{L^{\frac{12}{5}}}^{2}\|\mu\|_{L^{6}}+\|\n^{2}\ln\n\|_{L^{1}}).\ea\ee
Having   \eqref{3.30} and \eqref{3.30c} in hand,
we can make the following computation and estimate,
 \be\la{3.30d}\ba& \ep\si\int  \n_{0} c_{0} \mu -\ep\si\int \n \frac{\p f}{\p c} c \\
& \le C\si\ep\left( \|\mu\|_{L^{1}}+\|\n(\ln \n+1)\|_{L^{\frac{6}{5}}}\|c\|_{L^{6}}\right) \\
&\le  C\si\ep \left(1+ \|\na \mu\|_{L^{2}}\right)\left(1+\|\n\|_{L^{\frac{6}{5}}}\right)  + C\si\ep\left(1+\|\n \ln \n \|_{L^{\frac{6}{5}}}^{2}\right)\left(1+  \|\na c\|_{L^{2}} \right)\\
&\quad+  C\si \ep^{4}\left(1+\|\n \ln \n \|_{L^{\frac{6}{5}}}\right)( \|\n\|_{L^{\frac{12}{5}}}^{2}\|\mu\|_{L^{6}}+\|\n^{2}\ln\n\|_{L^{1}})\\
&\le C+\frac{1}{2} \|\na \mu\|_{L^{2}}^{2}+\frac{\ep}{4} \|\na c\|_{L^{2}}^{2}\\
&\quad+C \si\left(\ep\|\n\ln \n\|_{L^{\frac{6}{5}}}^{4}+\ep^{4}\|\n\ln \n\|_{L^{\frac{6}{5}}}\|\n^{2}\ln \n\|_{L^{1}}+\ep^{8}\|\n\ln \n\|_{L^{\frac{6}{5}}}^{4}\|\n\|_{L^{\frac{12}{5}}}^{4}\right).\ea\ee
Then, we  compute
\bnn\ba  & \si\left(\ep\|\n\ln \n\|_{L^{\frac{6}{5}}}^{4}+\ep^{4}\|\n\ln \n\|_{L^{\frac{6}{5}}}\|\n^{2}\ln \n\|_{L^{1}}+\ep^{8}\|\n\ln \n\|_{L^{\frac{6}{5}}}^{4}\|\n\|_{L^{\frac{12}{5}}}^{4}\right)\\
&\le C+\frac{\ep^{2}\si}{8}\|\n\ln \n\|_{L^{\frac{6}{5}}}^{8}+\frac{\ep^{2}\si}{8}\|\n\|_{L^{4}}^{4}+ \si\ep^{8} \|\n\ln \n\|_{L^{\frac{6}{5}}}^{4}\|\n\|_{L^{\frac{12}{5}}}^{4}\\
&\le C+\frac{\ep^{2}\si}{4}\|\n\|_{L^{4}}^{4}+ \si\ep^{8} \|\n\ln \n\|_{L^{\frac{6}{5}}}^{4}\|\n\|_{L^{\frac{12}{5}}}^{4}\\
&\le C+\frac{\ep^{2}\si}{4}\|\n\|_{L^{4}}^{4}+ \si\ep^{8} \|\n^{2}\|_{L^{6}}^{\frac{20}{11}}\\
&\le C(\de)+\frac{\ep^{2}\si}{2}\|\n\|_{L^{4}}^{4}+ \frac{\si\ep^{4}\de}{2} \|\na(\n^{2})\|_{L^{2}}^{2},\ea\enn
where in the third inequality sign we have used
\bnn \|\n\ln \n\|_{L^{\frac{6}{5}}}^{4}\|\n\|_{L^{\frac{12}{5}}}^{4}\le C\|\n\|_{L^{12}}^{\frac{38}{11}}\le C+\|\n^{2}\|_{L^{6}}^{\frac{20}{11}},\enn owing   to   interpolation and the fact $\|\n\|_{L^{1}}=m_{1}.$
By the above estimates, substituting   \eqref{3.31a}-\eqref{3.32a} and \eqref{3.30d}  into  \eqref{1.8} concludes
  \bnn \ba &\ep^{2}\si\|\n\|_{L^{4}}^{4}+ \|\na u \|_{L^{2}}^{2}+\|\na \mu\|_{L^{2}}^{2}+\ep\|\na c\|_{L^{2}}^{2} + \ep^{4}\si\int \left( \de\left|\na \n^{2}\right|^{2} + 4 |\na \sqrt{\n}|^{2}\right)\\& \le C+C\|\n\|_{L^{\frac{6}{5}}}^{2},\ea\enn
which, along with  \eqref{3.30} and \eqref{3.30c}, implies \be\la{a26} \ba& \ep^{2} \|\n\|_{L^{4}}^{4}+\| u \|_{H^{1}_{0}}^{2}+ \ep\|c\|_{H^{1}}^{2}+\|\mu\|_{H^{1}}^{2}  +\ep^{4}  \|\left|\na \n^{2}\right|+ |\na \sqrt{\n}|\|_{L^{2}}^{2}   \le   C+C\|\n\|_{L^{\frac{6}{5}}}^{2}.
\ea\ee

 \smallskip
 {{\it Step 3.}}
 By  \eqref{a26},  it  is clear that
 \be\la{a26av} \ba& \|\n\|_{L^{4}}^{4}+\| u \|_{H^{1}_{0}}^{2} +\|\mu\|_{H^{1}}^{2} + \|c\|_{H^{1}}^{2} +  \| \left|\na \n^{2}\right|+ |\na \sqrt{\n}|\|_{L^{2}}^{2} \le C(\ep).\ea\ee
From  \cite[Lemma 3.17]{novo} we take the Bogovskii operator
\be\la{05} \mathcal{B}=[ \mathcal{B}_{1}, \mathcal{B}_{2}, \mathcal{B}_{3}]:\,\,\,\left\{f\in L^{p}\,\,|\int f=0\right\}\mapsto W_{0}^{1,p}(\om),\quad p\in (1,\infty).\ee   Then, $ \div \mathcal{B}(f)=f$  a.e. in $\om$, and   moreover,
 \be\la{04}\|\na \mathcal{B}(f)\|_{L^{p}}\le C\|f\|_{L^{p}},\quad
  \|\mathcal{B}(f)\|_{L^{p}}\le C\|g\|_{L^{p}},\ee
  where  $f=\div g$ and  $g\in L^{p}$ with $g\cdot n|_{\p\om}=0$.
Furthermore,  we  write     $\eqref{n4}_{1}$ as  the equivalent form
\be\la{a30}  \ep^{4}\lap \n=\div (\n u+\ep^{2}\mathcal{B}(\n-\n_{0})).\ee
Applying Lemma \ref{lem2.3}  to   \eqref{a30},    using \eqref{a26av} and  \eqref{04},   we find
   \be\la{a31}\ba \|\na \n\|_{L^{4}}&\le\|\n u+ \ep^{2}\mathcal{B}(\n-\n_{0})\|_{L^{4}} \\
  &\le\|\n u\|_{L^{4}}  +\|\na \mathcal{B}(\n-\n_{0}) \|_{L^{4}} \\
  &\le \|u\|_{L^{6}} \|\n^{2}\|_{L^{6}}^{\frac{1}{2}}+\|\n-\n_{0}\|_{L^{4}}\le   C(\ep),\ea\ee
and hence,
 \be\la{a33} \ba\|\n \|_{H^{2}}&\le C\|\div (\n u+\ep^{2}\mathcal{B}(\n-\n_{0}))\|_{L^{2}}\\
 & \le  \|u\cdot\na \n +\n \div u\|_{L^{2}}  +\|\div \mathcal{B}(\n-\n_{0})\|_{L^{2}}  \le C(\ep).\ea\ee
Combining   \eqref{a26av} with   \eqref{a33}  gives
\bnn \|\sigma F^{1}(u,\mu,c)\|_{L^{\frac{3}{2}}}+\| \sigma F^{2}(u,\mu,c)\|_{L^{6}}+\|\sigma F^{3}(u,\mu,c)\|_{L^{6}}\le C(\ep).\enn
By   $L^{p}$ regularity estimates,  we obtain
\be\la{a13}  \|u\|_{W^{2,\frac{3}{2}}}+\| \mu\|_{W^{2,6}}+\| c\|_{W^{2,6}}\le C(\ep).\ee
From  \eqref{a13}  we have   $\|\si F^{1}(u,\mu,c)\|_{L^{6}}\le C(\ep)$, and thus $ \|u\|_{W^{2,6}}\le C(\ep).$
By  a bootstrap procedure,
\bnn\|(u,\mu,c)\|_{W^{2,p}}\le C(\ep),\quad \forall \,\,\, p\in (1,\infty).\enn This completes the proof of   Proposition \ref{c} and \eqref{a20}. \end{proof}

\bigskip

\section{$\ep$-Limit procedure for  the approximation solutions}% in Theorem \ref{t3.1}}

In this section, we shall take the $\ep$-limit procedure and prove the following result.

 \begin{theorem} \la{t4.1} Under the same   assumptions  as in Theorem \ref{t3.1},    the  system
 \begin{equation}\label{n6}
\left\{\ba
&\div (\n u)=0,\\
&\div(\n u\otimes u)+\na \left(\de \n^{4} + \n^{2} \frac{\p f}{\p \n}\right)
 =\div\mathbb{S}_{ns}+\n\mu \na c- \n \frac{\p f}{\p c} \na c+\n g_{1}+g_{2},\\
 &\n u \cdot \na c=\lap \mu,\\
&\n\mu= \n \frac{\p f}{\p c}-\lap c,\ea \right.
\end{equation}
with  the boundary conditions \eqref{1a}   admits a weak solution  $(\n,u,\mu,c)$  such that
\be\la{e7b}\int \n =m_{1}, \quad\int \n c =m_{2},\ee
 \be\ba\la{e7c} 0\le \n\in L^{5}(\om),\,\,\,u\in  H_{0}^{1}(\om,\r),\,\,\,(\mu,\,c)\in H^{1}(\om)\times H^{1}(\om).\ea\ee
% In addition,   the following  integral quantities   hold true:
Moreover, $\forall\,\,\,\Phi\in C_{0}^{\infty}(\om,\r)$ and $\phi\in C^{\infty}(\overline{\om})$,
\be\la{e5}\ba &\int \left(\de \n^{4} +\n^{2}\frac{\p f}{\p \n}\right)\div \Phi\\
&=\int \left(\mathbb{S}_{ns}-\n u\otimes u\right) :\na \Phi+\int \left(\n \frac{\p f}{\p c}\na c-\n\mu \na c-\n g_{1}-g_{2}\right)\cdot \Phi, %\quad \forall\,\,\,\Phi\in C_{0}^{\infty}(\om,\r),
\ea\ee and
\be\la{e4}\int \n u\cdot \na c \phi+\int \na \mu\cdot \na \phi=0,\quad \int  \n\mu\phi -\n \frac{\p f}{\p c}\phi=\int \na c\cdot \na \phi; %\quad \forall\,\,\,\phi\in C_{0}^{\infty}(\om);
\ee
when $(\n,u)$  is prolonged by zero outside $\Omega$,
 %After extended to   $\r\backslash \om$  by zero,  the function $(\n,u)$ satisfies
 \be\la{e7} \int_{\r}b(\n) u\cdot \na \phi=\int_{\r} \phi\left(b'(\n)\n-b(\n)\right)\div u, \ee
where $b(z)=z$, or  $b(z)\in C^{1}([0,\infty))$ with  $b'(z)=0$ if $z$ is large; and
the following energy inequality holds:
\be\la{e7a}\int \left(\lambda_{1}|\na u|^{2}+(\lambda_{1}+\lambda_{2})(\div u)^{2}+|\na \mu|^{2}\right)\le \int \left(\n g_{1}+g_{2}\right)\cdot u.\ee
\end{theorem}

   Theorem \ref{t4.1}  is indeed a result of   $\ep$-limit of the solutions  $(\n_{\ep},u_{\ep},\mu_{\ep},c_{\ep})$  obtained in  Theorem \ref{t3.1}, as shown below.  First the following lemma   derives     some  uniform  in $\ep$  estimates  on $(\n_{\ep},u_{\ep},\mu_{\ep},c_{\ep})$.
\begin{lemma}\la{lem4.1} Let    $ (\n_{\ep},u_{\ep},\mu_{\ep},c_{\ep})$ be  a solution    in Theorem \ref{t3.1}.  Then there exists a constant  $C$ which  is independent  of   $\ep,$ such that
 \be\la{b6}\ba &\|\n_{\ep}^{5}\|_{L^{1}}+\|\n_{\ep}^{3}\frac{\p f}{\p \n_{\ep}} \|_{L^{1}}+\|c_{\ep}\|_{H^{1}}\le C.\ea\ee
\end{lemma}

\begin{proof}   Let $\mathcal{B}$  be the Bogovskii operator  as defined in  \eqref{05}.  If we test    $\eqref{n1}_{2}$   by  $\mathcal{B}(\n_{\ep}-\n_{0})$, we infer
 \be\la{b5}\ba&\int\left(\de\n_{\ep}^{4}+\n_{\ep}^{2}\frac{\p f}{\p \n_{\ep}}\right)\n_{\ep} \\
&= \int \left(\de\n_{\ep}^{4}+\n_{\ep}^{2}\frac{\p f}{\p \n_{\ep}} \right)\n_{0}-\int  (\n_{\ep} g_{1}+g_{2})\cdot \mathcal{B}(\n_{\ep}-\n_{0})\\
&\quad + \ep^{2}\int\n_{\ep} u_{\ep} \cdot \mathcal{B}(\n_{\ep}-\n_{0})+\ep^{4}\int  \na \n_{\ep} \cdot \na u_{\ep}\mathcal{B}(\n_{\ep}-\n_{0})\\&\quad-\int \n_{\ep} u_{\ep}\otimes u_{\ep}:\na\mathcal{B}(\n_{\ep}-\n_{0})\\
&\quad+ \int \lambda_{1}(\na u_{\ep}+(\na u_{\ep})^{\top}):\na \mathcal{B}(\n_{\ep}-\n_{0})+  \lambda_{2}\div u_{\ep}\div \mathcal{B}(\n_{\ep}-\n_{0})\\
&\quad+ \int \left(\n_{\ep}\frac{\p f}{\p c_{\ep}}-\n_{\ep}\mu_{\ep} \right)\mathcal{B}(\n_{\ep}-\n_{0})\cdot \na c_{\ep} \\
&= \sum_{i=1}^{7}I_{i}.\ea\ee
Owing to  \eqref{b1}, \eqref{b2}, \eqref{a19}, \eqref{04},  and the simple fact $\n^{2}\frac{\p f}{\p \n}=  (\g-1)\n^{\g}+H_{1}\n,$  we get
\bnn\ba I_{1}+I_{2} &\le \left|\int \left(\de\n_{\ep}^{4}+\n_{\ep}^{2}\frac{\p f}{\p \n_{\ep}}\right)\n_{0}\right|+\left|\int  (\n_{\ep} g_{1}+g_{2})\mathcal{B}(\n_{\ep}-\n_{0})\right|\\
&\le C\int \left(\de \n_{\ep}^{4} +\n_{\ep}^{2}\frac{\p f}{\p \n_{\ep}}\right)+C(1+\|\n_{\ep}\|_{L^{\frac{6}{5}}})\|\na \mathcal{B}(\n_{\ep}-\n_{0})\|_{L^{2}} \\
&\le   \frac{1}{8}\int\left(\de\n_{\ep}^{5}+\n_{\ep}^{3}\frac{\p f}{\p \n_{\ep}}\right) +C.\ea\enn
Next, by \eqref{a26}  we have
 \bnn\ba &I_{3}+I_{4}+I_{5}\\
 &\le  \left(\ep^{2}\|\n_{\ep}\|_{L^{2}}+\ep^{4}\|\na \n_{\ep}\|_{L^{2}}\right)\|  u_{\ep}\|_{H^{1}}\|\mathcal{B}(\n_{\ep}-\n_{0})\|_{L^{\infty}}\\&\quad+\|\n_{\ep}\|_{L^{\frac{12}{5}}}\|u_{\ep}\|_{L^{6}}^{2}\|\na\mathcal{B}(\n_{\ep}-\n_{0})\|_{L^{4}} \\
&\le C \left(1+\|\n_{\ep}\|_{L^{\frac{6}{5}}}^{2}\right)\|\n_{\ep}\|_{L^{\frac{12}{5}}}\|\mathcal{B}(\n_{\ep}-\n_{0})\|_{W^{1,4}} \\
&\le \frac{\de}{8}\|\n_{\ep}\|_{L^{5}}^{5}+C.\ea\enn
Similarly,   \bnn\ba  I_{6}&\le \left| \int \lambda_{1}(\na u_{\ep}+(\na u_{\ep})^{\top}):\na \mathcal{B}(\n_{\ep}-\n_{0})+  \lambda_{2}\div u_{\ep}\div\mathcal{B}(\n_{\ep}-\n_{0}) \right|\\
&\le  C \|\na u_{\ep}\|_{L^{2}}\|\n_{\ep}-\n_{0}\|_{L^{2}} \le \frac{\de}{8}\|\n_{\ep}\|_{L^{5}}^{5}+C.\ea\enn
To deal with the last  term,  we  multiply  $\eqref{n1}_{4}$ by $c_{\ep}$, then use   \eqref{b1}, \eqref{b2},  \eqref{3.30c},  \eqref{a26} and the interpolation inequality to deduce
\be\la{78}\ba\|\na c_{\ep}\|_{L^{2}}^{2}&=\int \left(\n_{\ep}\mu_{\ep}- \n_{\ep} \frac{\p f}{\p c_{\ep}}\right) c_{\ep}\\
&\le  \| c_{\ep}\|_{L^{2}} \left(\int|\n_{\ep}\mu_{\ep}- \n_{\ep} \frac{\p f}{\p c_{\ep}}|^{2}\right)^{\frac{1}{2}}\\
&\le C\| c_{\ep}\|_{L^{2}} \left(\|\n_{\ep} \|_{L^{3}}\|\mu_{\ep}\|_{L^{6}}+\|\n_{\ep}\ln\n_{\ep} \|_{L^{2}}+1\right)\\
&\le \frac{1}{2}\|\na c_{\ep}\|_{L^{2}}^{2}+C\left( \|\n_{\ep}  \|_{L^{5}}^{\frac{5}{2}}+1\right),\ea\ee
 whence,
\bnn\ba I_{7}&\le \left|\int \left(\n_{\ep} \frac{\p f}{\p c_{\ep}}-\n_{\ep}\mu_{\ep} \right)\mathcal{B}(\n_{\ep}-\n_{0})\cdot\na c_{\ep}\right|\\
&\le  C \|\mathcal{B}(\n_{\ep}-\n_{0})\|_{L^{\infty}}\|\na c_{\ep}\|_{L^{2}} \left(\int|\n_{\ep}\mu_{\ep}- \n_{\ep} \frac{\p f}{\p c_{\ep}}|^{2}\right)^{\frac{1}{2}}\\
&\le   C\|\n_{\ep}\|_{L^{5}}^{4}+C.\ea\enn
In summary,   substituting   the estimates above back  into \eqref{b5},  using  \eqref{78},  we  conclude
 \bnn\|\na c_{\ep}\|_{L^{2}}+\int\left(\de\n_{\ep}^{5}+\n_{\ep}^{3}\frac{\p f}{\p \n_{\ep}}\right)\le C.\enn
This, along  with   \eqref{a26} and  \eqref{3.30c}, gives rise to \eqref{b6}.    The proof of Lemma \ref{lem4.1} is completed.

Having   \eqref{b1},  \eqref{b2},  \eqref{a26} and  \eqref{b6} in hand,   we  can take the limit as $\ep\to 0$  of  $(\n_{\ep}, u_{\ep},\mu_{\ep},c_{\ep})$, subject to  some subsequence, so that,
\be\la{b12}  \n_{\ep} \rightharpoonup \n \,\,{\rm in}\,\, L^{5}\cap L^{\g+1},  \quad \n_{\ep}^{4} \rightharpoonup \overline{\n^{4}}\quad {\rm in}\quad L^{\frac{5}{4}},\ee
\be\la{b10} (\na u_{\ep},\,\na \mu_{\ep},\na c_{\ep})\rightharpoonup  (\na u,\,\na \mu,\,\na c_{\ep})\,\,{\rm in}\,\,\,\, L^{2},\ee
\be\la{b11} (u_{\ep},\, \mu_{\ep},\,c_{\ep}) \rightarrow    (u,\,\mu,\,c)\,\,\,\,{\rm in}\quad L^{p_{1}}\,\,\,(1\le p_{1}<6), \ee
\be\la{b15}   \ep^{4} \na \n_{\ep} \rightarrow 0\,\,\,\, {\rm in}\,\,\,\, L^{2},\ee
\be\la{b15}   \ep^{2} \n_{\ep} \rightarrow 0,\,\,\,\,\,   \ep^{2} \n_{\ep} u_{\ep}\rightarrow 0,\,\,\,\,\, \ep \n_{\ep}c_{\ep}  \rightarrow 0,\,\,\, \ep^{4} \na \n_{\ep}\na u_{\ep}\rightarrow 0\,\,\,\, {\rm in}\quad  L^{1}.\ee
Moreover,  it follows from  \eqref{b12} and  \eqref{b11}   that
\be\la{b14}  (\n_{\ep}  u_{\ep},\,\n_{\ep}  \mu_{\ep})\rightharpoonup (\n u,\,\n\mu)\,\,\,  {\rm in}\,\,\,  L^{2}, \quad \quad  \n_{\ep}  u_{\ep} \otimes u_{\ep} \rightharpoonup \n u \otimes u \,\,\, {\rm in}\quad  L^{p}\,\,({\rm for\,\,some}\,\,p>1),\ee
and
\be\la{b18} \ba  \n_{\ep} \frac{\p f}{\p c_{\ep}}&=\n_{\ep}\ln\n_{\ep}H_{1}'(c_{\ep})+\n_{\ep}H_{2}'(c_{\ep}) \rightharpoonup \overline{\n\ln \n}H_{1}'(c)+\n H_{2}'(c)=\overline{ \n \frac{\p f}{\p c}}\,\, \,\,{\rm in}\,\,\,\, L^{2}, \ea\ee
\be\la{b18a}\ba  \n_{\ep}^{2} \frac{\p f}{\p \n_{\ep}}=(\g-1)\n_{\ep}^{\g}+\n_{\ep}H_{1}(c_{\ep}) \rightharpoonup (\g-1)\overline{\n^{\g}}+\n H_{1}(c)=\overline{ \n^{2} \frac{\p f}{\p \n}}\,\,\,\, {\rm in}\,\,\,\, L^{\frac{\g+1}{\g}}.\ea \ee

It remains   to verify     the strong convergence of $\na c_{\ep}$, i.e.,\be\la{b19} \na   c_{\ep} \rightarrow \na c\quad {\rm in}\quad  L^{2}.\ee
 {In fact, as in  \cite{Fei},  we use   equality $\eqref{n1}_{4}$  to obtain} 
\be\la{b16} \int  \na c_{\ep}\na \phi=\int \n_{\ep}\mu_{\ep}\phi-\int  \n_{\ep}\frac{\p f}{\p c_{\ep}} \phi,\ee
which, together with  \eqref{b11}, \eqref{b14}-\eqref{b18},   provides us
\bnn\ba\lim_{\ep \rightarrow 0}\int |\na c_{\ep}|^{2}
&=\lim_{\ep \rightarrow 0}\int \n_{\ep}\mu_{\ep}c_{\ep}-\lim_{\ep \rightarrow 0}\int \n_{\ep}\frac{\p f}{\p c_{\ep}} c_{\ep} =\int \n \mu c -\int \overline{\n \frac{\p f}{\p c}}c.\ea\enn
On the other hand,    if we    select  $\phi=c$ in \eqref{b16},  we obtain
\bnn\ba\int  |\na c|^{2}
=\lim_{\ep \rightarrow 0}\int \na c_{\ep}\cdot \na c
&=\lim_{\ep \rightarrow 0}\int \n_{\ep}\mu_{\ep}c-\lim_{\ep \rightarrow 0}\int  \n_{\ep}\frac{\p f}{\p c_{\ep}}  c =\int \n \mu c -\int \overline{\n \frac{\p f}{\p c}}c.\ea\enn
Thus  \bnn \lim_{\ep\rightarrow0}\int |\na c_{\ep}|^{2}=\int |\na c |^{2}, \enn
which, together with \eqref{b10},  guarantees  \eqref{b19}.

From \eqref{b12}-\eqref{b19},    we are able  to   pass limit  and get  the  integral  equalities \eqref{e5}-\eqref{e4}  with     $ \n^{4},\,\n \frac{\p f}{\p c},\,\n^{2}\frac{\p f}{\p\n}$ replaced by    $\overline{\n^{4}},\,\overline{\n \frac{\p f}{\p c}},\,\overline{\n^{2}\frac{\p f}{\p \n}}$, respectively.   In addition,   we obtain   \eqref{e7b} and      \eqref{e7a} from     \eqref{w3},  $\eqref{n1}_{1}$,   \eqref{w7},    and \eqref{b6}.

 Finally,    \eqref{e7}  is guaranteed by  the following    lemma, whose proof is available in \cite[Lemma 2.1]{novo1} and \cite[Lemma 3.3]{novo}.
 \begin{lemma}\la{lem4.2} Let  $(\n,u)$  be a solution to $\eqref{n6}_{1}.$
Assume that  $\n\in L^{2}(\om) $ and $ u\in  H^{1}_{0}(\om,\r)$.  If  we extend   $(\n,u)$   by zero outside $\om,$  we have     \be\la{wq} {\rm div}(b(\n) u) +(b'(\n)\n-b(\n))\div u=0\quad {\rm in}\quad \mathcal{D}'(\mathbb{R}^{3}),\ee
where $b(z)=z$, or  $b\in C^{1}([0,\infty))$ with $b'(z)=0$ for large $z$.
  \end{lemma}
%\begin{proof}
%  In case of    $b(\n)=\n$,  the proof is the same as   in \cite[Lemma 2.1]{novo1}.  Next,      mollifying  equality  $\eqref{n6}_{1}$  yields
%\be\la{sse}  0=\div(\n^{\eps} u)+\underbrace{(\div (\n u))^{\eps}-\div(\n^{\eps} u)}_{S_{1}^{\eps}},\quad a.e.\,\,{\rm in}\,\,\r.\ee
%By  the commutator estimate (see \cite[lemma 2.3]{lion}),  one has    $S_{1}^{\eps}\rightarrow0$ in $L^{1}.$    Therefore,
 % multiplying   \eqref{sse}  by $b'(\n^{\eps})$,  and using integration by parts,   we get  the desired \eqref{wq} by sending   $\eps\rightarrow 0$.
%\end{proof}

In order to   complete  the proof of  Theorem \ref{t4.1},   we need to verify  \be\la{b21} \overline{\n^{4}}=\n^{4},\quad \overline{\n \frac{\p f}{\p c} }=\n  \frac{\p f}{\p c},\quad \overline{\n^{2}  \frac{\p f}{\p \n}}=\n^{2}  \frac{\p f}{\p \n}.\ee
For that purpose,  let us define   \bnn\ba C^{2}([0,\infty))\ni b_{n}(\n) =\left\{\ba
&\n\ln (\n+\frac{1}{n}),\quad \quad \n\le n;\\
&(n+1)\ln (n+1+\frac{1}{n}),\,\,\n\ge n+1.\\
\ea \right.
\ea\enn
First we see that   $b_{n}(\n)\rightarrow \n\ln\n$ a.e.    because of the fact: $\n \in L^{1}$.   Select  $b_{n}$ in  \eqref{wq} and  send  $n\rightarrow \infty$   to  obtain
\bnn  {\rm div}(u\n\ln \n)+\n \div u =0\quad {\rm in}\quad \mathcal{D}'(\mathbb{R}^{3}).\enn This implies
\be\la{4.8}  \int  \n \div u=0.
 \ee
On the other hand,   multiplying $\eqref{n1}_{1}$  by  $b_{n}'(\n_{\ep})$ gives
\be\la{4.9}\ba &\int (b_{n}'(\n_{\ep})\n_{\ep}-b_{n}(\n_{\ep})) \div u_{\ep}\\
& =\ep^{2}\int  \n_{0}b_{n}'(\n_{\ep})-\ep^{2}\int  \n_{\ep} b_{n}'(\n_{\ep}) -\ep^{4}\int b_{n}''(\n_{\ep}) |\na  \n_{\ep}|^{2}\\
& \le\ep^{2}\int  \n_{0}b_{n}'(\n_{\ep})-\ep^{2}\int  \n_{\ep} b_{n}'(\n_{\ep}) {-\ep^{4}\int_{\{x:\,\,b_{n}''(\n_{\ep})\le  0\}} b_{n}''(\n_{\ep}) |\na  \n_{\ep}|^{2}.} \ea \ee
{It follows from  \eqref{a26av} that $\|\na \n_{\ep}\|_{L^{2}}\le C\left(\|\na \n_{\ep}^{2}\|_{L^{2}}+\|\na \sqrt{\n_{\ep}}\|_{L^{2}}\right)\le C(\ep)$. Then,  for fixed $\ep>0,$}
{\bnn\ba \left|-\ep^{4}\int_{\{x:\,b_{n}''(\n_{\ep})\le 0\}} b_{n}''(\n_{\ep}) |\na  \n_{\ep}|^{2}\right| &\le  C(\ep) \int_{\{x:\,b_{n}''(\n_{\ep})\le 0\}}  |\na  \n_{\ep}|^{2}\\
& \le C(\ep)  \int_{\{x:\,n\le \n_{\ep}\le n+1\}}  |\na  \n_{\ep}|^{2}\rightarrow0 \quad (n\rightarrow\infty),\ea\enn}
 {where  in  the second inequality   we have used  the fact  $b_{n}''(\n_{\ep})\ge 0 $ if $\n_{\ep}\le n$ or $ \n_{\ep}\ge n+1$.}

Recalling    \eqref{b6} and the definition of $b_{n}$,   one deduces
\bnn\ba &\lim_{n\rightarrow\infty} \int  \n_{0}b_{n}'(\n_{\ep})\\
&=\lim_{n\rightarrow\infty}\left( \int_{\{\n_{\ep}\le n\}}\n_{0}b_{n}'(\n_{\ep})+\int_{\{\n_{\ep}>n\}}\n_{0}b_{n}'(\n_{\ep}) \right)\\
&\le \lim_{n\rightarrow\infty} \int_{\{\n_{\ep}\le n\}} \n_{0}\left(\ln (\n_{\ep}+\frac{1}{n}) +\frac{\n_{\ep}}{\n_{\ep}+\frac{1}{n}}\right)+C\lim_{n\rightarrow\infty}{\rm meas}\,|\{x;\,\n_{\ep}\ge n\}|\\
&\le  \lim_{n\rightarrow\infty}\int_{\{1/2\le \n_{\ep}\le n\}} \n_{0} \ln (\n_{\ep}+\frac{1}{n})+\lim_{n\rightarrow\infty}\int \frac{\n_{0} \n_{\ep}}{\n_{\ep}+\frac{1}{n}} \\
&\le C.\ea\enn
 Similarly,  \bnn \lim_{n\rightarrow\infty} \int  \n_{\ep}b_{n}'(\n_{\ep})\le C.\enn   Therefore,  taking  sequentially $n\rightarrow\infty$ and   $\ep\rightarrow0$   in \eqref{4.9}, using   \eqref{4.8}, one has
  \be\la{4.10}\int \overline{\n \div u} = \lim_{\ep\rightarrow0} \int  \n_{\ep} \div u_{\ep}\le 0=\int  \n \div u.\ee

To  proceed,  define   the following effective viscous flux:
\bnn \mathbb{F}_{\ep} =\de \n^{4}_{\ep}+\n_{\ep}^{2} \frac{\p f}{\p  \n_{\ep}}-(2\lambda_{1}+\lambda_{2})\div u_{\ep}\quad {\rm and}\quad   \overline{\mathbb{F}}=\de\overline{\n^{4}}+\overline{\n^{2} \frac{\p f}{\p \n}}-(2\lambda_{1}+\lambda_{2})\div u.\enn
We have the following lemma.

\begin{lemma}\la{lem4.3} Under the assumptions    in  Theorem \ref{t4.1}, the following property holds:
\be\la{4.6}\ba  \lim_{\ep\rightarrow0}\int \phi   \n_{\ep} \mathbb{F}_{\ep}=\int \phi \n \overline{\mathbb{F}},\quad \forall\,\,\,\phi\in C_{0}^{\infty}(\om).\ea \ee
\end{lemma}

Let us   continue to prove \eqref{b21}  with the  aid of \eqref{4.6}.  The proof  of Lemma \ref{lem4.3} will be postponed to the end of this section.

In view of   \eqref{4.10},   $\mathbb{F}_{\ep}$  and $\overline{\mathbb{F}},$   we take  $ \phi \rightarrow 1$ in \eqref{4.6} and deduce  \be\la{4.11}\ba& \lim_{\ep\rightarrow0}\int \n_{\ep}\left(\de \n^{4}_{\ep}+
 \n_{\ep}^{2} \frac{\p f}{\p  \n_{\ep}} \right)\le \int \n \left(\de\overline{\n^{4}}+\overline{\n^{2} \frac{\p f}{\p\n}}\right).\ea \ee
According to  \eqref{b18a} and \eqref{4.11},    we have
 \bnn\ba & \int \left(\de\overline{\n^{5}}+(\g-1)\overline{\n^{\g+1}}+
\overline{\n^{2}}H_{1}(c) \right)\\
&=\lim_{\ep\rightarrow0}\int \n_{\ep} \left(\de\n_{\ep}^{4}+(\g-1)\n_{\ep}^{\g}+\n_{\ep} H_{1}(c_{\ep})\right)\\
  &=\lim_{\ep\rightarrow0}\int \n_{\ep}\left(\de \n_{\ep}^{4}+\n_{\ep}^{2}  \frac{\p f}{\p \n_{\ep}}\right)\\
  &\le \int \n \left(\de\overline{\n^{4}}+\overline{\n^{2} \frac{\p f}{\p \n}}\right)=\int \n \left(\de\overline{\n^{4}}+(\g-1)\overline{\n^{\g}} +\n H_{1}(c)\right),\ea \enn
which implies  \be\la{4.12a}\ba &\int  \de\left(\n\overline{\n^{4}}-  \overline{\n^{5}}\right) \ge  (\g-1)\int \left( \overline{\n^{\g+1}}-\n \overline{\n^{\g}}\right)+\int  \left(\overline{\n^{2}}-\n^{2} \right)H_{1}(c)\ge 0,\ea \ee
where the last inequality is  due  to  convexity and   $H_{1}(c)\ge0$.
Next,  for  the given   constant $\beta>0$ and  any $\eta\in C^{\infty}(\om),$
\bnn\ba
  0&\le \int  \left(\n_{\ep}^{4}-(\n+\beta\eta)^{4}\right)(\n_{\ep}-(\n+\beta\eta))\\
 &= \int  \left(\n_{\ep}^{5}-\n_{\ep}^{4}\n -\n_{\ep}^{4}\beta\eta-(\n+\beta\eta)^{4}\n_{\ep}+(\n+\beta\eta)^{5}\right).\ea \enn
By   \eqref{4.12a},  as  $\ep\rightarrow0$,
\bnn\ba 0&\le \int \left(\overline{\n^{5}}-\n \overline{\n^{4}}-\overline{\n^{4}} \beta\eta+(\n+\beta\eta)^{4} \beta\eta\right)\le \int  \left(-\overline{\n^{4}} +(\n+\beta\eta)^{4} \right)\beta\eta.\ea\enn
Replacing  $-\beta$ with $\beta$ in  the  argument above,  and then taking   $\beta\rightarrow0,$  we get
      \bnn\ba   \int \left(\n^{4}  -\overline{\n^{4}}\right) \eta=0.\ea \enn
This  implies  $\overline{\n^{4}} =\n ^{4}$, and thus   $\n_{\ep}\rightarrow \n$  a.e.  in $\om$   since   $\eta$ is arbitrary.
Moreover,  \eqref{b12} implies  that,   for all  $s\in [1,5)$,
\be\la{4.20c}   \n_{\ep}\rightarrow \n\quad  {\rm in}\quad L^{s}.\ee
As a result of     \eqref{4.20c}, \eqref{b12},  \eqref{b18}-\eqref{b18a},  we obtain \eqref{b21}.    The proof of  Theorem \ref{t4.1} is completed. \end{proof}

 {{\bf Proof of Lemma \ref{lem4.3}.}}  {We will   prove Lemma \ref{lem4.3} by the results developed in \cite{lion}.}   Let  $\lap^{-1}(h)=K*h$ be  the convolution of $h$ with the fundamental solution $K$ of the Laplacian in $\r$.  For  $\p_{i}\lap^{-1}\,\,\,(i=1,2,3)$, by  the Mikhlin multiplier theory (cf. \cite{stein}),
\be\la{e0}\left\{\ba & \|\p_{i}\lap^{-1}(h) \|_{W^{1,p}(\om)}\le C(\om,p)\|h\|_{L^{p}(\r)},\quad p\in (1,\infty),\\
& \|\p_{i}\lap^{-1}(h) \|_{L^{p^{*}}(\om)}\le C(\om,p)\|\p_{i}\lap^{-1}(h) \|_{W^{1,p}(\r)},\quad p^{*}=\frac{3p}{3-p},\,\,p<3,\\
& \|\p_{i}\lap^{-1}(h) \|_{L^{\infty}(\om)}\le C(\om,p)\|h\|_{L^{p}(\r)},\quad p>3.\ea\right.\ee
If $h_{n} \rightharpoonup h$ in $L^{p}(\r)$, we have \be\la{e1} \p_{j}\p_{i}\lap^{-1}(h_{n})  \rightharpoonup \p_{j}\p_{i}\lap^{-1}(h)\quad{\rm in }\,\,\,\,L^{p},\ee
and  additionally,  by the Rellich-Kondrachov compactness  theorem, \be\la{e2}  \p_{i}\lap^{-1}(h_{n})  \rightarrow \p_{i}\lap^{-1}(h)\quad {\rm in }\,\, L^{q}, \ee where $q<p^{*}$ if $p<3$ and $q\le \infty$ if $p>3.$

Prolonging  $\n_{\ep}$ to the whole space $\r$ by zero,
multiplying    $\eqref{n1}_{2}^{i}$ by $\phi\p_{i}\lap^{-1}(\n_{\ep})$ with $\phi\in C_{0}^{\infty}(\om)$,   we obtain
 \be\la{4.2}\ba  &\int  \phi \n_{\ep}\mathbb{F}_{\ep}\\
&=-\int \p_{i}\lap^{-1}(\n_{\ep}) \p_{i}\phi \left(\de\n_{\ep}^{4}+\n_{\ep}^{2} \frac{\p f}{\p \n_{\ep}} -(\lambda_{1}+\lambda_{2})\div u_{\ep}\right)\\
&\quad+\lambda_{1}\int \left( \p_{j}u^{i}_{\ep}\p_{i}\lap^{-1}(\n_{\ep}) \p_{j}\phi- u^{i}_{\ep}\p_{j}\p_{i}\lap^{-1}(\n_{\ep})\p_{j}\phi+  \n_{\ep} u_{\ep} \cdot\na \phi\right)\\
&\quad-\int  \left( \n_{\ep}\mu_{\ep} \p_{i} c_{\ep}- \n_{\ep}  \frac{\p f}{\p c_{\ep}}\p_{i} c_{\ep} \right)\phi\p_{i}\lap^{-1}(\n_{\ep})-\int\left(\n_{\ep}g_{1}+g_{2}\right)\phi  \p_{i}\lap^{-1}(\n_{\ep})\\
&\quad-\int \n_{\ep}u_{\ep}^{j}u_{\ep}^{i}\p_{j}\phi \p_{i}\lap^{-1}(\n_{\ep})-\int \n_{\ep}u_{\ep}^{j}u_{\ep}^{i}\phi\p_{j} \p_{i}\lap^{-1}(\n_{\ep})\\
&\quad+\ep^{2}\int \n_{\ep}u_{\ep}^{i} \phi  \p_{i}\lap^{-1}(\n_{\ep}) +\ep^{4}\int  \na \n_{\ep}\cdot \na u_{\ep}^{i}\phi\p_{i}\lap^{-1}(\n_{\ep}),
\ea\ee
where the second line on the right-hand side   comes from
\bnn \ba
&\lambda_{1}\int  \p_{j}u^{i}_{\ep}\left(\p_{i}\lap^{-1}(\n_{\ep})\p_{j}\phi+\p_{j}\p_{i}\lap^{-1}(\n_{\ep})\phi\right)\\
&=\lambda_{1}\int \left( \p_{j}u^{i}_{\ep}\p_{i}\lap^{-1}(\n_{\ep})\p_{j}\phi- u^{i}_{\ep}\p_{j}\p_{i}\lap^{-1}(\n_{\ep})\p_{j}\phi-  u_{\ep}^{i}\p_{i}\n_{\ep} \phi\right)\\
&=\lambda_{1}\int \left( \p_{j}u^{i}_{\ep}\p_{i}\lap^{-1}(\n_{\ep}) \p_{j}\phi- u^{i}_{\ep}\p_{j}\p_{i}\lap^{-1}(\n_{\ep})\p_{j}\phi+  \n_{\ep} u_{\ep} \cdot\na \phi\right)+\lambda_{1}\int \n_{\ep} \div  u_{\ep}\phi.\ea\enn

Next,    since $(\n_{\ep},u_{\ep})\in (H^{1}, H_{0}^{1})$, then $\div (\n_{\ep}u_{\ep})\in L^{\frac{3}{2}}(\r)$  and $\div (\n_{\ep}u_{\ep})=0$ in $\r\backslash\om$.  In addition,  $\n_{\ep}\in H^{2}$ and $ \frac{\p \n_{\ep}}{\p n}|_{\p\om}=0$ imply   \bnn\div (\textbf{1}_{\om}\na \n_{\ep})=\left\{\ba
&\lap \n_{\ep},\quad {\rm in}\,\,\,\om,\\
&0,\quad \r\setminus\om. \ea\right.\enn
Thus,  it makes sense to extend  $\eqref{n1}_{1}$ to the whole space by zero,
\bnn \ep^{2}(\n_{\ep}-\n_{0})+\div (\n_{\ep}u_{\ep})=\ep^{4}\div (\textbf{1}_{\om}\na \n_{\ep})\quad {\rm in}\quad \r,\enn
which yields by straight forward computations,
 \bnn\ba&-\int \n_{\ep}u^{i}_{\ep}\phi\p_{i}\p_{j}\lap^{-1}(\n_{\ep}u_{\ep}^{j} )\\
 &=-\int  \n_{\ep}u^{i}_{\ep}\phi\p_{i} \lap^{-1}(\div(\n_{\ep}u_{\ep}))\\
&=-\ep^{4}\int \n_{\ep}u^{i}_{\ep}\phi \p_{i}\lap^{-1}(\div(\textbf{1}_{\om} \na\n_{\ep}))+\ep^{2}\int \n_{\ep}u^{i}_{\ep}\phi \p_{i}\lap^{-1}(\n_{\ep}-\n_{0}),
\ea \enn and
\be\la{4.2q}\ba
&-\int  \n_{\ep}u_{\ep}^{j}u_{\ep}^{i}\p_{j}\phi \p_{i}\lap^{-1}(\n_{\ep})-\int \n_{\ep}u_{\ep}^{j}u_{\ep}^{i}\phi\p_{j} \p_{i}\lap^{-1}(\n_{\ep})\\
&=-\int \n_{\ep}u_{\ep}^{j}u_{\ep}^{i}\p_{j}\phi \p_{i}\lap^{-1}(\n_{\ep})
 +\int u_{\ep}^{i}\phi \left[\n_{\ep}\p_{i}\p_{j}\lap^{-1}(\n_{\ep}u^{j}_{\ep})-\n_{\ep}u_{\ep}^{j} \p_{j}\p_{i}\lap^{-1}(\n_{\ep}) \right]\\
 &\quad-\int \n_{\ep}u^{i}_{\ep}\phi\p_{i}\p_{j}\lap^{-1}(\n_{\ep}u_{\ep}^{j} )\\
&=-\int  \n_{\ep}u_{\ep}^{j}u_{\ep}^{i}\p_{j}\phi \p_{i}\lap^{-1}(\n_{\ep}) +\int u_{\ep}^{i}\phi \left[\n_{\ep}\p_{i}\p_{j}\lap^{-1}(\n_{\ep}u^{j}_{\ep})-\n_{\ep}u_{\ep}^{j} \p_{j}\p_{i}\lap^{-1}(\n_{\ep}) \right]\\
&\quad-\ep^{4}\int  \n_{\ep}u^{i}_{\ep}\phi \p_{i}\lap^{-1}(\div(\textbf{1}_{\om} \na\n_{\ep}))+\ep^{2}\int \n_{\ep}u^{i}_{\ep}\phi \p_{i}\lap^{-1}(\n_{\ep}-\n_{0}) .\ea\ee
Now,    replace the second line from the bottom in \eqref{4.2} by \eqref{4.2q}   to obtain
\be\la{4.3}\ba  &\int\phi \n_{\ep}\mathbb{F}_{\ep} \\&= -\int \p_{i}\lap^{-1}(\n_{\ep}) \p_{i}\phi \left(\de\n_{\ep}^{4}+\n_{\ep}^{2} \frac{\p f}{\p\n_{\ep}} -(\lambda_{1}+\lambda_{2})\div u_{\ep}\right) \\
&\quad +\lambda_{1}\int\left(  \p_{j}u^{i}_{\ep}\p_{i}\lap^{-1}(\n_{\ep}) \p_{j}\phi-  u^{i}_{\ep}\p_{j}\p_{i}\lap^{-1}(\n_{\ep})\p_{j}\phi+ \n_{\ep} u_{\ep} \cdot\na\phi\right) \\
&\quad -\int \left((\n_{\ep}\mu_{\ep} \p_{i} c_{\ep}+ \n_{\ep}  \frac{\p f}{\p c_{\ep}}\p_{i} c_{\ep} )\phi\p_{i}\lap^{-1}(\n_{\ep}) -  (\n g_{1}+g_{2})\phi  \p_{i}\lap^{-1}(\n_{\ep})\right)\\
&\quad-\int \n_{\ep}u_{\ep}^{j}u_{\ep}^{i}\p_{j}\phi \p_{i}\lap^{-1}(\n_{\ep})+\int u_{\ep}^{i}\phi \left[\n_{\ep}\p_{i}\p_{j}\lap^{-1}(\n_{\ep}u^{j}_{\ep})-\n_{\ep}u_{\ep}^{j} \p_{j}\p_{i}\lap^{-1}(\n_{\ep}) \right]\\
&\quad-\ep^{4}\int   \n_{\ep}u^{i}_{\ep}\phi \p_{i}\lap^{-1}(\div(\textbf{1}_{\om} \na\n_{\ep}))- \na \n_{\ep}\cdot \na u_{\ep}^{i}\phi\p_{i}\lap^{-1}(\n_{\ep})\\
&\quad+\ep^{2}\int  \n_{\ep}u^{i}_{\ep}\phi \p_{i}\lap^{-1}(2\n_{\ep}-\n_{0})\\
&= \sum_{i=1}^{7}J_{i}^{\ep},\ea\ee
where $J_{i}^{\ep}$ denotes the $i^{th}$ integral quantity on the right hand side of \eqref{4.3}.

On the other  hand,  if    we  take  $\ep$-limit in $\eqref{n1}_{2}$  first   and then multiply  the resulting  equation by  $\phi\p_{i}\lap^{-1}(\n)$, we obtain
 \be\la{4.5}\ba & \int \phi \n \overline{\mathbb{F}}\\
&= -\int  \p_{i}\lap^{-1}(\n) \p_{i}\phi \left(\de\overline{\n^{4}}+\overline{\n^{2} \frac{\p f}{\p \n}}-(\lambda_{1}+\lambda_{2})\div u \right) \\
&\quad +\lambda_{1}\int\left( \p_{j}u^{i} \p_{i}\lap^{-1}(\n) \p_{j}\phi-  u^{i} \p_{j}\p_{i}\lap^{-1}(\n)\p_{j}\phi+  \n  u  \cdot\na \phi\right)\\
&\quad -\int\left( (\n \mu  \p_{i} c  + \overline{\n  \frac{\p f}{\p c}}\p_{i} c)\phi \p_{i}\lap^{-1}(\n) -  (\n g_{1}+g_{2})\phi  \p_{i}\lap^{-1}(\n)\right) \\
&\quad  - \int \n u ^{j}u ^{i}\p_{j}\phi \p_{i}\lap^{-1}(\n ) +\int u ^{i}\phi\left[\n \p_{i}\p_{j}\lap^{-1}(\n u^{j})-\n u ^{j} \p_{j}\p_{i}\lap^{-1}(\n )\right]\\
&=\sum_{i=1}^{5}J_{i}. \ea\ee
In terms of \eqref{4.3} and \eqref{4.5},  to  prove \eqref{4.6}  it suffices to check
\bnn\lim_{\ep\rightarrow0} J_{i}^{\ep}=J_{i}\,\,(i=1, 2,  \cdots, 5)\quad {\rm and }\quad\lim_{\ep\rightarrow0} J_{i}^{\ep}=0\,\,(i=6, 7).\enn
In fact,  owing to \eqref{e2}, \eqref{b12}-\eqref{b10}, \eqref{b18a},    we  have   $\lim_{\ep\rightarrow0} J^{\ep}_{1}=J_{1}.$ In a similar  way,   for $i=2,3,4,$   we obtain    $\lim_{\ep\rightarrow0} J^{\ep}_{i}=J_{i}$ from
 \eqref{e1}-\eqref{e2},   \eqref{b12}-\eqref{b11},  \eqref{b14}-\eqref{b18},  and \eqref{b19}.    Next,   by  \eqref{e0}, \eqref{a26},  and \eqref{b6},  we estimate
\bnn\ba &|J_{6}^{\ep}+J_{7}^{\ep}|\\&\le \ep^{4}\|\na \n_{\ep}\|_{L^{2}}\|\n_{\ep}\|_{L^{3}}\|u_{\ep}\|_{L^{6}}+\ep^{2}\|  \n_{\ep}\|_{L^{2}}\|u_{\ep}\|_{L^{2}}\|\p_{i}\lap^{-1}(2\n_{\ep}-\n_{0})\|_{L^{\infty}}\\
&\quad+\ep^{4}\|  \na \n_{\ep}\|_{L^{2}}\|\na u_{\ep}\|_{L^{2}}\|\p_{i}\lap^{-1}(\n_{\ep})\|_{L^{\infty}}\\
&\le \ep^{2}\left(\ep^{2}\|\na \n_{\ep}\|_{L^{2}}\right) \|u_{\ep}\|_{H^{1}_{0}}\left( \|\n_{\ep}\|_{L^{3}}+\|\p_{i}\lap^{-1}(\n_{\ep})\|_{L^{\infty}}\right)\\
&\quad+\ep^{2}\|  \n_{\ep}\|_{L^{2}}\|u_{\ep}\|_{L^{2}}\|\p_{i}\lap^{-1}(2\n_{\ep}-\n_{0})\|_{L^{\infty}}\\
&\le C\ep\rightarrow 0 \quad {\rm as}\quad \ep\rightarrow 0.\ea\enn

Finally,  in  order to check $J_{5},$  we present   the following div-curl Lemma.
\begin{lemma}[\cite{fei5}]\la{lem4.4}  Let $\frac{1}{r_{1}}+\frac{1}{r_{2}}=\frac{1}{r}$ and $1\le r,r_{1},r_{2}<\infty.$
 Suppose that
\bnn v_{\ep}\rightharpoonup v\,\,{\rm in}\,\,L^{r_{1}}\quad {\rm and}\quad w_{\ep}\rightharpoonup w\,\,{\rm in}\,\,L^{r_{2}}.\enn
Then,  for $i,j=1,2,3,$ \bnn v_{\ep}\p_{i}\p_{j}\lap^{-1}(w_{\ep})-w_{\ep}\p_{i}\p_{j}\lap^{-1}(v_{\ep})\rightharpoonup v\p_{i}\p_{j}\lap^{-1}(w)-w\p_{i}\p_{j}\lap^{-1}(v)\quad{\rm in}\,\,L^{r}.\enn
\end{lemma}

Taking    $v_{\ep}=\n_{\ep}u_{\ep}^{j}$ and $w_{\ep}=\n_{\ep}$ in  Lemma \ref{lem4.4}, and using \eqref{b11}, \eqref{b14},    we get   $\lim_{\ep\rightarrow0} J^{\ep}_{5}=J_{5}.$  This completes the proof of Lemma \ref{lem4.3}.

 \bigskip

  \section{Vanishing artificial pressure}

  In this section, we take the $\delta$-limit in the artificial pressure and prove the main result in Theorem \ref{t}.

By $\eqref{n6}_{1},$   it follows from     $\eqref{e4}_{1}$ that
\be\la{e8}\ba \int \na \mu\cdot \na \phi=-\int  \n  u \cdot \na c  \phi=-\int \n u \cdot \na (c \phi)+\int  \n  c u  \cdot \na \phi=\int  \n   c  u  \cdot \na \phi.\ea\ee
{Next,  from \eqref{b1}, \eqref{3.30}, and \eqref{e7c}, one can easily check that $c\in W^{2,p}(\om)$  for some $p>1$. Let $\phi=\na c \cdot \Phi$ in $\eqref{e4}_{2}$ with $\Phi\in C_{0}^{\infty}(\om;\r)$.}  We have,  by approximation if necessary,  \bnn\ba \int (\n \mu -\n  \frac{\p f}{\p c})  \na c  \cdot\Phi =\int \na c  \cdot \na (\na c \cdot\Phi) =  -\int \mathbb{S}_{c} :\na \Phi,\ea\enn
which, along with \eqref{e5},  leads  to
\be\la{e9}\ba\int \left(\de \n ^{4} +\n^{2} \frac{\p f}{\p \n} \right)\div \Phi =\int \left(\mathbb{S}_{ns}+\mathbb{S}_{c}-\n  u \otimes u\right):\na \Phi -\int (\n g_{1}+g_{2})\cdot\Phi.\ea\ee

As a result of  \eqref{e8},  \eqref{e9}, and Theorem \ref{t4.1},   we have the following theorem.

  \begin{theorem} \la{t5.1}
  Under the same   conditions  in Theorem \ref{t4.1},  for any fixed $\de>0,$  the following system
   \begin{equation}\label{n7}
\left\{\ba
&\div(\n u)=0,\\
&\div  (\n u\otimes u)+\na \left(\de \n^{4} + \n^{2}  \frac{\p f}{\p\n}\right)
 = \div\left(\mathbb{S}_{ns}+\mathbb{S}_{c} \right)+\n g_{1}+g_{2},\\
 &\div(\n u  c)=\lap \mu,\\
&\n\mu= \n  \frac{\p f}{\p c}-\lap c,\ea \right.
\end{equation}
with  the boundary conditions \eqref{1a}  admits a  weak solution $(\n_{\de},u_{\de},\mu_{\de},c_{\de})$ which satisfies \eqref{e7b} and \eqref{e7c}.
  \end{theorem}

We will  prove Theorem \ref{t} by taking $\de\rightarrow 0$ in the solutions     $(\n_{\de},u_{\de},\mu_{\de},c_{\de})$  obtained in Theorem \ref{t5.1}.
Firstly,   we derive some  refined   estimates on  $(\n_{\de},u_{\de},\mu_{\de},c_{\de})$ which are uniform  in $\de$.

\begin{lemma}\la{lem5.2} Let $(\n_{\de},u_{\de},\mu_{\de},c_{\de})$  be a solution obtained in Theorem \ref{t5.1}.  Assume  that \eqref{ss} is satisfied.    Then there is some $p>\frac{6}{5}$ and $\th>0$ with $\g+\th>2$ such that       \be\la{ss1} \de \|\n_{\de}^{4+\th}\|_{L^{1}}+\|\n_{\de}^{2+\th} \frac{\p f}{\p \n_{\de}}\|_{L^{1}}+\|u_{\de}\|_{H^{1}_{0}}+\|\mu_{\de}\|_{H^{1}}+\|c_{\de}\|_{W^{2,p}}\le C,\ee
where, and in what follows,  the   constant   $C$ is independent of $\de$.
    \end{lemma}

  \noindent {\it Proof.}   We shall borrow  some  ideas  from  \cite{freh,mpm} to give a
 weighted estimate on  pressure.
 Owing to \eqref{0} and   \eqref{ss},  it follows from  \eqref{e7a} that
\be\la{r1}\ba  \lambda_{1}\int |\na u_{\de}|^{2} + \int |\na \mu_{\de}|^{2}
&\le  \int \left(\n_{\de} g_{1} +g_{2}\right) \cdot u_{\de}\\&  \le \textbf{1}_{g_{1}}\|\n_{\de}u_{\de}\|_{L^{1}}\|g_{1}\|_{L^{\infty}}+\|u_{\de}\|_{H_{0}^{1}}\|g_{2}\|_{L^{\frac{6}{5}}},\ea\ee where $\textbf{1}_{g_{1}}=1$ if $\na \times g_{1}\neq 0$ and $\textbf{1}_{g_{1}}=0$ if $\na \times g_{1}=0.$  Taking \be\la{z10} b\ge\frac{3s-2}{s}\,\,\,{\rm and }\,\,\,s\in   \left[1,\frac{6(\g+\th)}{5\g+2\th}\right],\ee  we have \be\la{r18w}\ba  \|\n_{\de}u_{\de}\|_{L^{s}}^{s}&=\int  \left(\n_{\de}^{b}|u_{\de}|^{2}\right)^{\frac{7s-6}{6b-4}}\left( |u_{\de}|^{6}\right)^{\frac{2+(b-3)s}{6b-4}}\left(\n_{\de} \right)^{\frac{(6-s)b-4s}{6b-4}}\\
  &\le  \|\n_{\de}^{b}|u_{\de}|^{2}\|_{L^{1}}^{\frac{7s-6}{6b-4}}\| u_{\de}\|_{L^{6}}^{\frac{6(2+(b-3)s)}{6b-4}}\|\n_{\de} \|_{L^{1}}^{\frac{(6-s)b-4s}{6b-4}}.
 \ea\ee
Substituting   \eqref{r18w}     into   \eqref{r1}  gives rise to
 \be\la{r3}\ba &\|u_{\de}\|_{H_{0}^{1}}+\|\na \mu_{\de}\|_{L^{2}} \le C \left(\|\n_{\de}^{b}|u_{\de}|^{2}\|_{L^{1}}^{\frac{\textbf{1}_{g_{1}}}{6b-2}}+1\right).\ea\ee

{\underline {\it The case of  $\na \times g_{1}\neq 0$.}}
%The  detailed  process
The estimate is divided into several steps.

{\it Step 1.} Let $\mathcal{B}$ be the Bogovskii operator defined in \eqref{05}.
Choosing       $\Phi=\mathcal{B}(\n_{\de}^{\th}-|\om|^{-1}\int_{\om}\n_{\de}^{\th})$  in \eqref{e9}  with  $\th=\th(\g)$   small and  to be determined, we get
\be\la{q5.3}\ba  & \int \left(\de \n_{\de}^{4}+\n_{\de}^{2} \frac{\p f}{\p \n_{\de}}\right)\n_{\de}^{\th}\\
&= \left(|\om|^{-1}\int \n_{\de}^{\th}\right)\int \left(\de \n_{\de}^{4}+\n_{\de}^{2} \frac{\p f}{\p \n_{\de}}\right) + \int \mathbb{S}_{ns}:\na \Phi- (\n_{\de} g_{1}+g_{2})\cdot\Phi\\
&\quad-\int \n_{\de} u_{\de}\otimes u_{\de} :\na \Phi +\int \mathbb{S}_{c} :\na \Phi\\
&= \sum_{i=1}^{4}K_{i}.
\ea\ee
Firstly,  by  \eqref{b1} and \eqref{b2},
we have  \be\la{q5.10}\ba
 K_{1}  &\le C\int \left(\de \n_{\de}^{4}+\n_{\de}^{2} \frac{\p f}{\p \n_{\de}}\right)\\
 &\le \frac{1}{8}\int  \left(\de \n_{\de}^{4+\th}+(\g-1)\n_{\de}^{\g+\th}+\n_{\de}^{1+\th}H_{1}(c_{\de})\right)+C\\
&=\frac{1}{8} \int \left(\de \n_{\de}^{4}+\n_{\de}^{2} \frac{\p f}{\p \n_{\de}}\right)\n_{\de}^{\th}
+C.\ea\ee
Secondly, thanks to \eqref{r3} and \eqref{ss},
 \be\la{q5.13}\ba K_{2}&=\int \mathbb{S}_{ns}:\na \Phi- (\n_{\de} g_{1}+g_{2})\cdot\Phi\\
 &\le \left(\|\na u_{\de}\|_{L^{2}}+\|\n_{\de}\|_{L^{\frac{6}{5}}}\|g_{1}\|_{L^{\infty}}+\|g_{2}\|_{L^{\frac{6}{5}}}\right)\|\na \Phi\|_{L^{2}}\\
 &\le C\left(\|\n_{\de}^{b}|u_{\de}|^{2}\|_{L^{1}}^{\frac{\textbf{1}_{g_{1}}}{6b-2}}+ \|\n_{\de}\|_{L^{\frac{6}{5}}}+1\right)\|\n_{\de}^{\th}\|_{L^{2}}\\
 &\le  \frac{1}{8}\int_{\om}\n_{\de}^{2+\th} \frac{\p f}{\p\n_{\de}}+C\|\n_{\de}^{b}|u_{\de}|^{2}\|_{L^{1}}^{\frac{1}{6b-2}}+C. \ea\ee
Next, by \eqref{r3}, one has  %direct  calculation shows
 \be\la{x2}\ba \|\n_{\de}|u_{\de}|^{2}\|_{L^{t}}^{t}&=\int \left(\n_{\de}^{b}|u_{\de}|^{2}\right)^{\frac{4t-3}{(3b-2)}}\left(|u_{\de}|^{6}\right)^{\frac{1-t(2-b)}{(3b-2)}}\n_{\de}^{\frac{3b-t(2+b)}{(3b-2)}}\\
 &\le \|\n_{\de}^{b}|u_{\de}|^{2}\|_{L^{1}}^{\frac{4t-3}{(3b-2)}}\|u_{\de}\|_{H_{0}^{1}}^{\frac{6(1-t(2-b))}{(3b-2)}}\|\n_{\de}\|_{L^{1}}^{\frac{3b-t(2+b)}{(3b-2)}}\\
 &\le C \left(1+\|\n_{\de}^{b}|u_{\de}|^{2}\|_{L^{1}}^{\frac{5t-3}{(3b-1)}} \right).\ea\ee
Let   $t=\frac{\g+\th}{\g}$ in \eqref{x2}, then we have
 \be\la{q5.12}\ba K_{3}&=- \int \n_{\de} u_{\de}\otimes u_{\de} :\na \Phi\\
&\le\|\na \Phi\|_{L^{\frac{\g+\th}{\th}}} \|\n_{\de}|u_{\de}|^{2}\|_{L^{\frac{\g+\th}{\g}}}\\
&\le C\|\n_{\de}\|_{L^{\g+\th}}^{\th}\|\n_{\de}|u_{\de}|^{2}\|_{L^{\frac{\g+\th}{\g}}} \\
&\le \frac{1}{8}\int_{\om}\n_{\de}^{2+\th}\frac{\p f}{\p \n_{\de}}+C\left(1+\|\n_{\de}^{b}|u_{\de}|^{2}\|_{L^{1}}^{\frac{2\g+5\th}{\g(3b-1)}} \right).\ea\ee
Finally,  if we replace $u_{\de}$ with $\mu_{\de}$ and take $\bar{b}=3-\frac{2}{s}$ in \eqref{r18w},   we find
\be\la{r18a}\ba \|\n_{\de} \mu_{\de}\|_{L^{s}}^{s}
  &\le  C \|\n_{\de}^{\bar{b}}\mu_{\de}^{2}\|_{L^{1}}^{\frac{7s-6}{6\bar{b}-4}}.\ea\ee
 %  where  we have used
 %\bnn \|\mu_{\de}\|_{L^{6}}   \le C\left(1+\|\n_{\de}^{b}|u_{\de}|^{2}\|_{L^{1}}^{\frac{\textbf{1}_{g_{1}}}{6b-2}}\right)\left(1 +  \|\n_{\de}\|_{L^{2}}\right), \enn which comes from   \eqref{3.30} and    \eqref{r3}. Taking  ,  one has
Taking  $s=\frac{6(\g+\th)}{5\g+2\th}$ in \eqref{r18a},  we deduce
   \be\la{q5.9ab} \ba   \|\na^{2} c_{\de}\|_{L^{\frac{6(\g+\th)}{5\g+2\th}}}^{\frac{6(\g+\th)}{5\g+2\th}}
&\le C \|\lap c_{\de}\|_{L^{\frac{6(\g+\th)}{5\g+2\th}}}^{\frac{6(\g+\th)}{5\g+2\th}} \\
&\le C \|\n_{\de}\mu_{\de}+\n_{\de} \frac{\p f}{\p c_{\de}}\|_{L^{\frac{6(\g+\th)}{5\g+2\th}}}^{\frac{6(\g+\th)}{5\g+2\th}} \\
&\le C+C\|\n_{\de}^{\bar{b}}  \mu_{\de}^{2}\|_{L^{1}}^{\frac{3(\g+\th)}{5\g+2\th}} +C\|\n_{\de}\ln \n_{\de}\|_{L^{\frac{6(\g+\th)}{5\g+2\th}}}^{\frac{6(\g+\th)}{5\g+2\th}},\ea\ee where the exponents in the last inequality are due to  \be\la{bb5} s=\frac{6(\g+\th)}{5\g+2\th}\quad{\rm and}\quad\bar{b}=\frac{4\g+7\th}{3(\g+\th)}.\ee
  With the help of  \eqref{q5.9ab} and $\|\n_{\de}\|_{L^{1}}=m_{1},$ we have the following estimate,
    \be\la{q5.14}\ba K_{4}&\le  C\|\na \Phi\|_{L^{\frac{\g+\th}{\th}}}\|\na c_{\de}\|_{L^{\frac{2(\g+\th)}{\g}}}^{2}\\
 %&\le C\|\n_{\de}\|_{L^{(\g+\th)}}^{\th} \|\na c_{\de}\|_{L^{\frac{2(\g+\th)}{\g}}}^{2}\\
 &\le C\|\n_{\de}\|_{L^{(\g+\th)}}^{\th} \left(\|\na^{2} c_{\de}\|_{L^{\frac{6(\g+\th)}{5\g+2\th}}}^{\frac{6(\g+\th)}{5\g+2\th}}\right)^{\frac{5\g+2\th}{3(\g+\th)}}\\
 &\le  C+\frac{1}{8}\int \n_{\de}^{2+\th} \frac{\p f}{\p \n_{\de}} +C\|\n_{\de}^{\bar{b}}\mu_{\de}^{2}\|_{L^{1}}^{\frac{\g+\th}{\g}}.\ea\ee
In conclusion,  substituting \eqref{q5.10}-\eqref{q5.12},   \eqref{q5.14} back into \eqref{q5.3} gives rise to
 \be\la{r7}\ba \int \left(\de \n_{\de}^{4}+\n_{\de}^{2} \frac{\p f}{\p \n_{\de}}\right)\n_{\de}^{\th}
 & \le C+C \|\n_{\de}^{\bar{b}}|u_{\de}|^{2}\|_{L^{1}}^{\frac{2\g+5\th}{\g(3\bar{b}-1)}}+C \|\n_{\de}^{\bar{b}}\mu_{\de}^{2}\|_{L^{1}}^{\frac{\g+\th}{\g}}\\
 %&\le C+C \|\n_{\de}^{\bar{b}}(|u_{\de}|+|\mu_{\de}|)^{2}\|_{L^{1}}^{\frac{\g+\th}{\g}}\\
 &\le C+C \|\n_{\de}^{\bar{b}}(|u_{\de}|^{2}+\mu_{\de}^{2})\|_{L^{1}}^{\frac{\g+\th}{\g}},\ea\ee
from  \eqref{bb5} and    $\frac{2\g+5\th}{\g(3\bar{b}-1)}\le\frac{2\g+5\th}{\g(3\bar{b}-2)}= \frac{\g+\th}{\g}$ . %have  been used.
\medskip

 {\it Step 2.} We show the following estimate:

\begin{pro} \la{c1}For  any fixed $\a_{0}\in (0,1)$ and $x^{*}\in \overline{\om},$ there  is some  constant  $C$    independent of $\de$ or $x^{*}$,   such that    \be\la{r5}\ba  \int \frac{\n_{\de}^{\g}(x)}{|x-x^{*}|^{\a_{0}}}dx\le C\left( 1+\|\n_{\de}^{\bar{b}}(|u_{\de}|^{2}+\mu_{\de}^{2})\|_{L^{1}}\right),\ea\ee
with $\bar{b}$ being   defined in \eqref{bb5}.
\end{pro}

\begin{proof}  We consider two cases.

 {{\it Case 1:  boundary point  $x^{*}\in \p\om.$}}
As in \cite{freh}, we introduce
\be\la{r10} \xi^{i}(x)=\phi(x)\p_{i}\phi(x)\left(\phi(x)+|x-x^{*}|^{\frac{2}{2-\a_{0}}}\right)^{-\a_{0}},\quad i=1, 2, 3,\ee
where the function $\phi(x)\in C^{2}(\overline{\om})$ satisfies the following properties:
\be\la{x7}\left\{\ba& \phi(x)>0\,\,\,{\rm in}\,\,\om\,\,\,{\rm and}\,\,\,\phi(x)=0\,\,\,{\rm on}\,\,\,\p\om,\\
&|\phi(x)|\ge k_{1} \,\,\,{\rm if}\,\,\,x\in \om\,\,\,{\rm and}\,\,\,dist(x,\,\p\om)\ge k_{2},\\
& |\na\phi(x)|\ge k_{1} \,\,\,{\rm if}\,\,\,x\in \om\,\,\,{\rm and}\,\,\,dist(x,\,\p\om)\le k_{2}, \ea\right.\ee
and the  constants $k_{i}>0$ are given.
\begin{remark}
The function $\phi(x)$ satisfying \eqref{x7} is in fact the distance function near the boundary with $C^{2}$ extension to the whole $\om.$ Moreover, for every  point $x\in \om$ near the boundary, there is a unique $\tilde{x}\in \p\om$ such that
\be\la{bb} \na \phi=\frac{x-\tilde{x}}{\phi(x)}\quad {\rm and}\quad \phi(x)=|x-\tilde{x}|.\ee
See, e.g., \cite[Exercise 1.15]{zie} for the detail.\end{remark}
 It is clear  that $\xi\in L^{\infty}(\om)$ and $\xi=0$ on  $\p\om.$  In addition, a direct computation yields
\be\la{r15} \ba\p_{j}\xi^{i}
 &=\frac{\phi\p_{j}\p_{i}\phi}{\left(\phi+|x-x^{*}|^{\frac{2}{2-\a_{0}}}\right)^{ \a_{0}}}+\frac{  \p_{j}\phi\p_{i}\phi}{\left(\phi+|x-x^{*}|^{\frac{2}{2-\a_{0}}}\right)^{\a_{0}}}\\
 &\quad- \a_{0}\frac{\phi\p_{i}\phi\p_{j}\phi}{\left(\phi+|x-x^{*}|^{\frac{2}{2-\a_{0}}}\right)^{\a_{0}+1}}- \a_{0}\frac{\phi\p_{i}\phi\p_{j}|x-x^{*}|^{\frac{2}{2-\a_{0}}}}{\left(\phi+|x-x^{*}|^{\frac{2}{2-\a_{0}}}\right)^{\a_{0}+1}}.\ea\ee
Thus,   $ |\na \xi|\in L^{q}$ for all $q\in [2,\frac{3}{\a_{0}})$ because
$|\p_{j}\xi^{i}|   \le C+C|x-x^{*}|^{-\a_{0}}.$ Due to   \eqref{x7}   and  $\frac{2}{2-\a_{0}}>1$, the following inequalities hold true:
\be\la{z009}   \phi<\phi+|x-x^{*}|^{\frac{2}{2-\a_{0}}} \le  C|x-x^{*}|.\ee
 With   \eqref{x7}-\eqref{z009},  one deduces that, for $dist (x,\,\p\om)\le k_{2},$
\be\la{x8}\ba \div \xi&\ge -C + \frac{1}{2(1-\a_{0})}\frac{|\na \phi|^{2}}{\left(\phi+|x-x^{*}|^{\frac{2}{2-\a_{0}}}\right)^{\a_{0}}} \ge -C + \frac{C}{|x-x^{*}|^{\a_{0}}}.\ea\ee
Take  $\Phi=\xi$ in \eqref{e9} to obtain
 \be\la{r11}\ba&\int \left(\de \n_{\de}^{4} +\n_{\de}^{2} \frac{\p f}{\p \n_{\de}}\right)\div \xi+ \int \n_{\de} u_{\de}\otimes u_{\de} :\na  \xi\\&=\int \mathbb{S}_{ns} :\na  \xi -\int \left(\n_{\de}g_{1}+g_{2}\right)\cdot \xi+\int \mathbb{S}_{c} :\na  \xi.\ea\ee

The  first two terms on the right-hand side of \eqref{r11}  satisfy
\be\la{r13}\ba &\left| \int \mathbb{S}_{ns} :\na  \xi -\int \left(\n_{\de}g_{1}+g_{2}\right)\cdot \xi\right|  \le C(\a_{0}) \left(\|\na u_{\de} \|_{L^{2}}+1\right).\ea\ee
%Next,  if \be\la{bb6} \g>\bar{b}\ge\frac{5}{3},\ee
%then, H$\ddot{o}$lder inequality shows
 %\bnn \|\n_{\de}\mu_{\de} \|_{L^{\frac{3}{2}}}^{2}\le \|\n_{\de}^{\bar{b}}|\mu_{\de}|^{2} \|_{L^{1}}\left(\int \n_{\de}^{6-3\bar{b}}\right)^{\frac{1}{3}}\le C\|\n_{\de}^{\bar{b}}|\mu_{\de}|^{2} \|_{L^{1}}\enn
 %and \bnn \|\n_{\de}\ln \n_{\de} \|_{L^{\frac{3}{2}}}^{2}\le  C + C\int\n_{\de}^{2} \frac{\p f}{\p \n_{\de}}.\enn
Next, let \be\la{bb7}\frac{\g+\th}{\th}<\frac{3}{\a_{0}},\ee where $\th$ and $\a_{0}$ will  be determined in  \eqref{ref4}.  One deduces \be\la{r13w}\ba  \left| \int \mathbb{S}_{c} :\na  \xi\right| & \le  C\|\na \xi \|_{L^{\frac{\g+\th}{\th}}} \|\na c_{\de} \|_{L^{\frac{2(\g+\th)}{\g}}}^{2} \le C\|\na^{2} c_{\de} \|_{L^{\frac{6(\g+\th)}{5\g+2\th}}}^{2}\\
 &\le C+C\|\n_{\de}^{\bar{b}}\mu_{\de}^{2} \|_{L^{1}} +C \int\n_{\de}^{2} \frac{\p f}{\p \n_{\de}},\ea\ee
 where the last inequality is from  \eqref{q5.9ab}.

Now let us focus on the left-hand side of \eqref{r11}.  Owing to  \eqref{x8},   one has   \be\la{r12}\ba
 \int \left(\de \n_{\de}^{4} +\n_{\de}^{2} \frac{\p f}{\p \n_{\de}}\right)\div \xi\ge -C \int \left(\de \n_{\de}^{4} +\n_{\de}^{2}\frac{\p f}{\p \n_{\de}}\right)+C \int_{\om\cap B_{k_{2}}(x^{*})} \frac{\left(\de \n_{\de}^{4} +\n_{\de}^{2}\frac{\p f}{\p \n_{\de}}\right)}{|x-x^{*}|^{\a_{0}}}.\ea\ee
By \eqref{bb}, one has $\p_{j}\p_{i}\phi=\frac{\p_{i}(x-\tilde{x})^{j}}{\phi}-\frac{\p_{j}\phi\p_{i}\phi}{\phi}$. Then,
\bnn \ba&\int \frac{\phi \n_{\de}  u_{\de}\otimes u_{\de} : (\p_{j}\p_{i}\phi)_{3\times3}}{\left(\phi+|x-x^{*}|^{\frac{2}{2-\a_{0}}}\right)^{ \a_{0}}} =\int \frac{  \n_{\de}  |u_{\de}|^{2}}{\left(\phi+|x-x^{*}|^{\frac{2}{2-\a_{0}}}\right)^{ \a_{0}}}-\int \frac{  \n_{\de}  |u_{\de}\cdot\na \phi|^{2}}{\left(\phi+|x-x^{*}|^{\frac{2}{2-\a_{0}}}\right)^{ \a_{0}}}, \ea\enn
and hence, by  \eqref{r15} and  \eqref{z009},  we have
\be\la{x9}\ba &\int \n_{\de} u_{\de}\otimes u_{\de} :\na  \xi\\
&=\int \frac{  \n_{\de}  |u_{\de}|^{2}}{\left(\phi+|x-x^{*}|^{\frac{2}{2-\a_{0}}}\right)^{ \a_{0}}} - \a_{0} \int \frac{\phi \n_{\de}(u_{\de}\cdot \na\phi)^{2}}{\left(\phi+|x-x^{*}|^{\frac{2}{2-\a_{0}}}\right)^{\a_{0}+1}}\\
&\quad - \a_{0}\int \frac{\phi\n_{\de} (u_{\de}\cdot \na |x-x^{*}|^{\frac{2}{2-\a_{0}}})(u_{\de}\cdot\na \phi)}{\left(\phi+|x-x^{*}|^{\frac{2}{2-\a_{0}}}\right)^{\a_{0}+1}}\\
&\ge (1- \a_{0})\int \frac{  \n_{\de}  |u_{\de}|^{2}}{\left(\phi+|x-x^{*}|^{\frac{2}{2-\a_{0}}}\right)^{ \a_{0}}} - \a_{0}\int \frac{\phi\n_{\de} (u_{\de}\cdot \na |x-x^{*}|^{\frac{2}{2-\a_{0}}})(u_{\de}\cdot\na \phi)}{\left(\phi+|x-x^{*}|^{\frac{2}{2-\a_{0}}}\right)^{\a_{0}+1}}\\
&\ge \frac{(1- \a_{0})}{2}\int \frac{  \n_{\de}  |u_{\de}|^{2}}{\left(\phi+|x-x^{*}|^{\frac{2}{2-\a_{0}}}\right)^{ \a_{0}}} -C(\a_{0})\int \frac{\phi^{2}\n_{\de}|u_{\de}|^{2}  |x-x^{*}|^{\frac{2\a_{0}}{2-\a_{0}}}}{\left(\phi+|x-x^{*}|^{\frac{2}{2-\a_{0}}}\right)^{\a_{0}+2}}\\ &\ge C \int_{\om\cap B_{k_{2}}(x^{*})} \frac{  \n_{\de}  |u_{\de}|^{2}}{ |x-x^{*}|^{\a_{0}}} -C\|\n_{\de}|u_{\de}|^{2}\|_{L^{1}}.\ea\ee
%where the  inequalities are due to     \eqref{z009}  and the Cauchy inequality.
%
Therefore, \eqref{r11} together with \eqref{r13} and \eqref{r13w}-\eqref{x9} yield
%inequalities \eqref{r13} and \eqref{r13w}-\eqref{x9} ensure that \eqref{r11} satisfies
\be\la{x10}\ba &
\int_{\om\cap B_{k_{2}}(x^{*})}\left(\frac{\left(\de \n_{\de}^{4} +\n_{\de}^{2} \frac{\p f}{\p \n_{\de}}\right)}{|x-x^{*}|^{\a_{0}}}+\frac{  \n_{\de}|u_{\de}|^{2}}{|x-x^{*}|^{\a_{0}}}\right) \\
& \le C \int \left(\de \n_{\de}^{4} +\n_{\de}^{2} \frac{\p f}{\p \n_{\de}}\right) +C\left(\|u_{\de}\|_{H_{0}^{1}}+\|\n_{\de}^{\bar{b}}\mu_{\de}^{2}\|_{L^{1}}+
\|\n_{\de}|u_{\de}|^{2}\|_{L^{1}}+1\right)\\
& \le C(\g,\th,\overline{H})\left( \int \left( \de \n_{\de}^{4} +\n_{\de}^{2} \frac{\p f}{\p \n_{\de}} \right)\n_{\de}^{\th} \right)^{\frac{\g}{\g+\th}}\\&\quad+C\left(\|u_{\de}\|_{H_{0}^{1}}+\|\n_{\de}^{\bar{b}}\mu_{\de}^{2}\|_{L^{1}}+\|\n_{\de}|u_{\de}|^{2}\|_{L^{1}}+1\right)\\
 & \le C +C\|\n_{\de}^{\bar{b}}(|u_{\de}|^{2}+\mu_{\de}^{2})\|_{L^{1}},\ea\ee
 where, for  the last two inequalities  we have used  the  H\"older inequality,   \eqref{r3},  \eqref{r7} as well as
 \bnn \|\n_{\de}|u_{\de}|^{2}\|_{L^{1}}\le C+C\|\n_{\de}^{\bar{b}}|u_{\de}|^{2}\|_{L^{1}}^{\frac{2}{3\bar{b}-1}} \le C+C\|\n_{\de}^{\bar{b}}|u_{\de}|^{2}\|_{L^{1}},\enn  which comes from  \eqref{x2}.

 { {\it Case 2:  interior  point $x^{*}\in \om$}.}
There is a constant $r>0$ such that $dist (x^{*},\,\p\om)=3r$.  Let $\chi$ be a smooth cut-off function  satisfying $ \chi= 1$  in  $B_{r}(x^{*})$ and  $\chi=0$
 outside   $B_{2r}(x^{*})$,  as well as $|\na \chi|\le 2r^{-1}.$ Choosing  $\Phi(x)=\frac{x-x^{*}}{|x-x^{*}|^{\a_{0}}}\chi^{2}$  in $\eqref{e9}$,    we find
  \be\la{r8}\ba&\int \left(\de \n_{\de}^{4} +\n_{\de}^{2} \frac{\p f}{\p\n_{\de}}\right)  \frac{3-\a_{0}}{|x-x^{*}|^{\a_{0}}}\chi^{2} + \int \n_{\de} u_{\de}\otimes u_{\de} :\na  \left(\frac{x-x^{*}}{|x-x^{*}|^{\a_{0}}}\chi^{2}\right)\\
&= -\int \left(\n_{\de}g_{1}+g_{2}\right)\cdot \frac{x-x^{*}}{|x-x^{*}|^{\a_{0}}}\chi^{2}+\int \mathbb{S}_{ns} :\na  \left(\frac{x-x^{*}}{|x-x^{*}|^{\a_{0}}}\chi^{2}\right)\\
&\quad  -2\int \left(\de \n_{\de}^{4} +\n_{\de}^{2} \frac{\p f}{\p \n_{\de}}\right)  \chi \frac{\na\chi\cdot(x-x^{*})}{|x-x^{*}|^{\a_{0}}}+\int  \mathbb{S}_{c} :\na  \left(\frac{x-x^{*}}{|x-x^{*}|^{\a_{0}}}\chi^{2}\right).\ea\ee
By a direct computation, one has
\bnn\ba&\p_{i} \left(\frac{x^{j}-(x^{*})^{j}}{|x-x^{*}|^{\a_{0}}}\chi^{2}\right)\\
&= \frac{\p_{i}(x^{j}-(x^{*})^{j})}{|x-x^{*}|^{\a_{0}}}\chi^{2} -\a_{0} \frac{(x^{j}-(x^{*})^{j})(x^{i}-(x^{*})^{i})}{|x-x^{*}|^{\a_{0}+2}}\chi^{2}+2\chi \frac{x^{j}-(x^{*})^{j}}{|x-x^{*}|^{\a_{0}}}\p_{i} \chi\\
&\in L^{q},\quad q\in [2,\frac{3}{\a_{0}}),\ea\enn
and hence, the second term on the left-hand side of \eqref{r8} satisfies
 \be\la{x3} \ba& \int \n_{\de} u_{\de}\otimes u_{\de} :\na \left(\frac{x-x^{*}}{|x-x^{*}|^{\a_{0}}}\chi^{2}\right)\\
 &\ge(1-\a_{0})\int \frac{\n_{\de} |u_{\de}|^{2}}{|x-x^{*}|^{\a_{0}}}\chi^{2}+2\int  \frac{ \chi\n_{\de} (u_{\de}\cdot\na \chi )(u_{\de}\cdot(x-x^{*}))}{|x-x^{*}|^{\a_{0}}}\\
 &\ge\frac{1-\a_{0}}{2}\int \frac{\n_{\de} |u_{\de}|^{2}}{|x-x^{*}|^{\a_{0}}}\chi^{2}-C \int_{r<|x-x^{*}|< 2r} \frac{\n_{\de}|u_{\de}|^{2}}{|x-x^{*}|^{\a_{0}}},\ea\ee where $C$ is independent of $r.$

 For the terms on the right-hand side   of \eqref{r8}, we have
 \bnn \ba &\left| -\int \left(\n_{\de}g_{1}+g_{2}\right)\cdot \frac{x-x}{|x-x^{*}|^{\a_{0}}}\chi^{2}+\int  \mathbb{S}_{ns} :\na  \left(\frac{x-x^{*}}{|x-x^{*}|^{\a_{0}}}\chi^{2}\right)\right|\\
 &\le C+C\|u_{\de}\|_{H_{0}^{1}} \ea\enn
 and
\bnn\ba&  -2\int \left(\de \n_{\de}^{4} +\n_{\de}^{2} \frac{\p f}{\p \n_{\de}}\right)  \chi \frac{\na\chi\cdot(x-x^{*})}{|x-x^{*}|^{\a_{0}}}\le  C\int_{r<|x-x^{*}|<2r} \frac{\left(\de \n_{\de}^{4} +\n_{\de}^{2}\frac{\p f}{\p\n_{\de}}\right)}{|x-x^{*}|^{\a_{0}}},\ea \enn
where $C$ is independent of $r.$
Similarly  to  \eqref{r13w}, we deduce
 \bnn \ba\int  \mathbb{S}_{c} :\na  \left(\frac{x-x^{*}}{|x-x^{*}|^{\a_{0}}}\chi^{2}\right)
 &\le C  +C\|\n_{\de}^{\bar{b}}\mu_{\de}^{2} \|_{L^{1}} +C\int\n_{\de}^{2} \frac{\p f}{\p \n_{\de}}.\ea\enn
From  the  above estimates,   we obtain   \be\la{r9*}\ba &\int_{B_{r}(x^{*})}\left(\frac{\left(\de \n_{\de}^{4} +\n_{\de}^{2} \frac{\p f}{\p \n_{\de}}\right)}{ |x-x^{*}|^{\a_{0}}}+ \frac{\n_{\de}|u_{\de}|^{2}}{|x-x^{*}|^{\a_{0}}}\right) \\
& \le C \int \n_{\de}^{2} \frac{\p f}{\p \n_{\de}} +C\left(\|u_{\de}\|_{H_{0}^{1}}+\|\n_{\de}^{\bar{b}}\mu_{\de}^{2}\|_{L^{1}}+\|\n_{\de}|u_{\de}|^{2}\|_{L^{1}}+1\right)\\
&\quad+  C \int_{r<|x-x^{*}|< 2r} \left(\frac{\left(\de \n_{\de}^{4} +\n_{\de}^{2} \frac{\p f}{\p \n_{\de}}\right)}{ |x-x^{*}|^{\a_{0}}}+ \frac{\n_{\de}|u_{\de}|^{2}}{|x-x^{*}|^{\a_{0}}}\right) \\
 & \le  C +C\|\n_{\de}^{\bar{b}}(|u_{\de}|^{2}+\mu_{\de}^{2})\|_{L^{1}} \\&\quad+ C \int_{r<|x-x^{*}|< 2r} \left(\frac{\left(\de \n_{\de}^{4} +\n_{\de}^{2} \frac{\p f}{\p \n_{\de}}\right)}{ |x-x^{*}|^{\a_{0}}}+ \frac{\n_{\de}|u_{\de}|^{2}}{|x-x^{*}|^{\a_{0}}}\right),\ea\ee
 where  the last inequality follows from \eqref{x10}.

We need to discuss two situations: $(i)$  $x^{*}\in \om$ is far from the boundary.  $(ii)$  $x^{*}\in \om$ is close to the boundary.

$(i)$  The case of $dist(x^{*},\,\p\om)= 3r\ge \frac{k_{2}}{2}>0$, where $k_{2}$ is the same as in \eqref{x7}.
From  \eqref{r9*}, one has
 \be\la{r9a}\ba &\int_{B_{r}(x^{*})}\frac{\left(\de \n_{\de}^{4} +\n_{\de}^{2} \frac{\p f}{\p \n_{\de}}\right)}{ |x-x^{*}|^{\a_{0}}} \\
 & \le  C +C\|\n_{\de}^{\bar{b}}(|u_{\de}|^{2}+\mu_{\de}^{2})\|_{L^{1}}+C(k_{2}) \int \left(\de \n_{\de}^{4} +\n_{\de}^{2} \frac{\p f}{\p \n_{\de}}+\n_{\de}|u_{\de}|^{2}\right)\\
 & \le C \left(1+\|\n_{\de}^{\bar{b}}(|u_{\de}|^{2}+\mu_{\de}^{2})\|_{L^{1}}\right),\ea\ee
 where  in  the last inequality we have also used \eqref{x10}.

$(ii)$  The case of  $dist(x^{*},\,\p\om)= 3r<\frac{k_{2}}{2}$. Let $ |x^{*}-\tilde{x}^{*}|=dist(x^{*},\,\p\om).$  Then,
       \be\la{bb1} 4 |x-x^{*}|\ge |x-\tilde{x}^{*}|,\quad \forall \,\,\, x\notin B_{r}(x^{*}).\ee
   In view  of  \eqref{bb1}, we infer from \eqref{r9*} that \be\la{r9b}\ba &\int_{B_{r}(x^{*})}\left(\frac{\left(\de \n_{\de}^{4} +\n_{\de}^{2} \frac{\p f}{\p \n_{\de}}\right)}{ |x-x^{*}|^{\a_{0}}}+ \frac{\n_{\de}|u_{\de}|^{2}}{|x-x^{*}|^{\a_{0}}}\right) \\
 & \le  C+C\|\n_{\de}^{\bar{b}}(|u_{\de}|^{2}+\mu_{\de}^{2})\|_{L^{1}}+C \int_{r<|x-x^{*}|< 2r} \left(\frac{\left(\de \n_{\de}^{4} +\n_{\de}^{2} \frac{\p f}{\p \n_{\de}}\right)}{ |x-x^{*}|^{\a_{0}}}+ \frac{\n_{\de}|u_{\de}|^{2}}{|x-x^{*}|^{\a_{0}}}\right)\\
 & \le C +C\|\n_{\de}^{\bar{b}}(|u_{\de}|^{2}+\mu_{\de}^{2})\|_{L^{1}} +  C \int_{\om\cap B_{k_{2}}(\tilde{x}^{*})} \left(\frac{\left(\de \n_{\de}^{4} +\n_{\de}^{2} \frac{\p f}{\p \n_{\de}}\right)}{ |x-\tilde{x}^{*}|^{\a_{0}}}+ \frac{\n_{\de}|u_{\de}|^{2}}{|x-\tilde{x}^{*}|^{\a_{0}}}\right)\\
&\le C \left(1+\|\n_{\de}^{\bar{b}}(|u_{\de}|^{2}+\mu_{\de}^{2})\|_{L^{1}}\right),\ea\ee
where for the last inequality we have also used \eqref{x10}.

 In summary, we obtain \eqref{r5}  from \eqref{x10},   \eqref{r9a} and \eqref{r9b}.
 \end{proof}
\begin{remark}The case that  $x^*\in\om$ is close to the boundary was first treated by  Mucha-Pokorn\'{y}-Zatorska \cite{mpm}, where they   combined the test functions for both the interior and boundary cases.  \end{remark}

 {\it Step 3.}
By  \eqref{r5}  and   the H\"older inequality,  we have
  \be\la{z2}\ba
\int \frac{\n_{\de}^{\bar{b}}}{|x-x^{*}|}
&\le \left(\int \frac{\n_{\de}^{\g}}{|x-x^{*}|^{\a_{0}}}\right)^{\frac{\bar{b}}{\g}}\left(\int \frac{1}{|x-x^{*}|^{\frac{\g-\bar{b}\a_{0}}{\g-\bar{b}}}}\right)^{\frac{\g-\bar{b}}{\g}}\\
&\le C \left(\int \frac{\n_{\de}^{\g}}{|x-x^{*}|^{\a_{0}}} \right)^{\frac{\bar{b}}{\g}}\\
&\le C \left(1+ \|\n_{\de}^{\bar{b}}(|u_{\de}|^{2}+\mu_{\de}^{2})\|_{L^{1}}\right)^{\frac{\bar{b}}{\g}},\ea\ee
 if    \be\la{z8} \frac{\bar{b}(3-\a_{0})}{2}<\g.\ee
 We note that \eqref{z8} implies $\frac{\g-\bar{b}\a_{0}}{\g-\bar{b}}<3.$

Consider the Neumann  boundary value problem:
\be\la{z1}  \left\{\ba &\lap h(x^{*})=\n_{\de}^{\bar{b}}-\frac{1}{|\om|}\int_{\om}\n_{\de}^{\bar{b}}=\n_{\de}^{\bar{b}}-\mathfrak{m}\quad{\rm in}\,\,\, \om, \\
&\frac{\p h(x^{*})}{\p n}=0\,\,\,  {\rm on}\,\,\, \p\om,\ea\right.  \ee
 where $\mathfrak{m}=\frac{1}{|\om|}\int_{\om}\n_{\de}^{\bar{b}}.$\,\,
Recalling the  Green's function representation   $$h(x^{*})=\int_{\om}G(x^{*},x)\left(\n_{\de}^{\bar{b}}(x)-\mathfrak{m}\right)dx,$$
and using \eqref{z2}, we have
 \be\la{z3}\ba \|h\|_{L^{\infty}}&\le \sup_{x^{*}\in \om}\int_{\om}\frac{\left(\n_{\de}^{\bar{b}}(x)-\mathfrak{m}\right)}{|x-x^{*}|}dx\\
 &\le \sup_{x^{*}\in \om}\int_{\om}\frac{\n_{\de}^{\bar{b}}(x)}{|x-x^{*}|}dx+C \mathfrak{m}\\
 &\le C \left(1+  \|\n_{\de}^{\bar{b}}(\mu_{\de}^{2}+|u_{\de}|^{2})\|_{L^{1}}^{\frac{\bar{b}}{\g}}\right).\ea\ee
From  \eqref{z1}  one has
\bnn\ba \|\n_{\de}^{\bar{b}}\mu_{\de}^{2}\|_{L^{1}} &=  \int \mu_{\de}^{2}(\mathfrak{m}+\lap h)\\
&=\mathfrak{m}\int \mu_{\de}^{2}-2\int \mu_{\de}\na \mu_{\de}\cdot\na  h\\
&\le \mathfrak{m}\int \mu_{\de}^{2}+2\|\na \mu_{\de}\|_{L^{2}}\left(\int \mu_{\de}^{2}|\na h|^{2}\right)^{\frac{1}{2}},\ea\enn
 and
 \bnn\ba \int \mu_{\de}^{2}|\na h|^{2}&=-\int \left(\mu_{\de}^{2}h\lap h+2\mu_{\de}\na \mu_{\de}h\na h\right)\\
 &\le \|h\|_{L^{\infty}}\left(\mathfrak{m}\int \mu_{\de}^{2}+ \int \n_{\de}^{\bar{b}}\mu_{\de}^{2}+2\|\na \mu_{\de}\|_{L^{2}}\left(\int \mu_{\de}^{2}|\na h|^{2}\right)^{\frac{1}{2}}\right),\ea\enn
thus,
\be\la{ref2}\ba \|\n_{\de}^{\bar{b}}\mu_{\de}^{2}\|_{L^{1}}  \le C \left(\|h\|_{L^{\infty}}\|\na \mu_{\de}\|_{L^{2}}^{2}+\mathfrak{m}\|\mu_{\de}\|_{L^{2}}^{2}\right) .\ea\ee
Thanks to the  interpolation inequality and $\|\n_{\de}\|_{L^{1}}=m_{1}$, we have
\bnn \mathfrak{m}\le C\|\n_{\de}^{\bar{b}}\|_{L^{1}}\le C\|\n_{\de}^{\g}\|_{L^{1}}^{\frac{\bar{b}-1}{\g-1}} \quad{\rm and}\quad \|\n_{\de}\|_{L^{\frac{6}{5}}}^{2}\le C\|\n_{\de}^{\g}\|_{L^{1}}^{\frac{1}{3(\g-1)}}.\enn
Then, by \eqref{3.30} one has
  \bnn  \mathfrak{m}\|\mu_{\de}\|_{L^{2}}^{2} \le C \|\n_{\de}^{\g}\|_{L^{1}}^{\frac{3\bar{b}-2}{3(\g-1)}} (1+\|\na\mu_{\de}\|_{L^{2}}^{2}).\enn
Substituting  it back into \eqref{ref2} yields
  \bnn \ba  \|\n_{\de}^{\bar{b}}\mu_{\de}^{2}\|_{L^{1}}  \le C \left(  \|h\|_{L^{\infty}}+ \|\n_{\de}^{\g}\|_{L^{1}}^{\frac{3\bar{b}-2}{3(\g-1)}}\right)(1+\|\na\mu_{\de}\|_{L^{2}}^{2}).\ea \enn
Similarly, we have
 \bnn  \|\n_{\de}^{\bar{b}}|u_{\de}|^{2}\|_{L^{1}}\le C\|h\|_{L^{\infty}}\|\na u_{\de}\|_{L^{2}}^{2}.\enn
Thus, from    \eqref{z3}, \eqref{r5} and \eqref{r3},   we obtain
 \be\la{z6}\ba \|\n_{\de}^{\bar{b}}(|u_{\de}|^{2}+\mu_{\de}^{2})\|_{L^{1}} &\le C \left(  \|h\|_{L^{\infty}}+ \|\n_{\de}^{\g}\|_{L^{1}}^{\frac{3\bar{b}-2}{3(\g-1)}}\right)\left(1+\|\na \mu_{\de}\|_{L^{2}}+\|\na u_{\de}\|_{L^{2}}\right)^{2} \\
 &\le C +C \|\n_{\de}^{\bar{b}}(|u_{\de}|^{2}+\mu_{\de}^{2})\|_{L^{1}}^{\beta},\ea\ee with
 \bnn  \beta=\max\left\{\frac{\bar{b}}{\g}  ,\,\,\,\frac{3\bar{b}-2}{3(\g-1)}\right\}+\frac{1}{3\bar{b}-1} .\enn

 {\it Step 4.}
If  we can prove
  \be\la{z7}\ba \|\n_{\de}^{\bar{b}}(|u_{\de}|^{2}+\mu_{\de}^{2})\|_{L^{1}} \le C,\ea\ee
then,  we conclude   \eqref{ss1} from \eqref{r3}, \eqref{q5.9ab},  and \eqref{r7},  and thus complete the proof of
   Lemma \ref{lem5.2}.

To prove \eqref{z7},   it suffices to show  $\beta<1$ in view of  \eqref{z6}.  By  \eqref{bb7},    we may take $\th$    close  to zero as  $\a_{0}\rightarrow 0.$  From   \eqref{z8} and \eqref{bb5} we see that \be\la{ref4}
 \frac{\bar{b}}{\g}<\frac{2}{3-\a_{0}} \rightarrow \frac{2}{3}\quad ({\rm as}\,\,\a_{0}\rightarrow0)\quad {\rm and}\quad  \frac{1}{3\bar{b}-1} =\frac{\g+\th}{3\g+6\th}\rightarrow \frac{1}{3}\,\,\, ({\rm as}\,\,\th\rightarrow0).\ee
 Hence,   $\beta= \frac{\bar{b}}{\g} +\frac{1}{3\bar{b}-1}<1$ if  both $\a_{0}$ and $\th$ are chosen small enough. Besides,  to guarantee \eqref{z8}, from \eqref{bb5} we have   \bnn \g> \frac{3-\a_{0}}{2}\bar{b}=\frac{3-\a_{0}}{2}\cdot \frac{4\g+7\th}{3(\g+\th)}\rightarrow 2 \quad ({\rm as}\,\,\a_{0},\,\th\rightarrow0).\enn
 If $\beta=\frac{3\bar{b}-2}{3(\g-1)} +\frac{1}{3\bar{b}-1}$ (we have no need  checking \eqref{z8} any more), we see that $\frac{3\bar{b}-2}{3(\g-1)}<\frac{2}{3}$ is equivalent
  to $\g>\frac{3}{2} \bar{b}$.   By \eqref{bb5},
 \be\la{ref7} \g>\frac{3}{2} \bar{b}=\frac{3}{2}\cdot \frac{4\g+7\th}{3(\g+\th)}\rightarrow 2\quad ({\rm as}\,\,\th\rightarrow0).\ee
 This and \eqref{ref4} guarantee that $\beta<1$ as long as $\th$ is small.
\medskip

{\underline {\it The case of  $\na \times g_{1}= 0$.}}  In this case,  from \eqref{r3} we have
\bnn \|u_{\de}\|_{H_{0}^{1}}+\|\na \mu_{\de}\|_{L^{2}} \le C. \enn
Then,  the same deduction as \eqref{z6} yields
 \be\la{z6t}\ba \|\n_{\de}^{\bar{b}}(|u_{\de}|^{2}+\mu_{\de}^{2})\|_{L^{1}}
 &\le C \left( \|h\|_{L^{\infty}}+ \|\n_{\de}^{\g}\|_{L^{1}}^{\frac{3\bar{b}-2}{3(\g-1)}}\right) \\
 &\le C +C \|\n_{\de}^{\bar{b}}(|u_{\de}|^{2}+\mu_{\de}^{2})\|_{L^{1}}^{\beta},\ea\ee with
 \bnn \beta=\max\left\{\frac{\bar{b}}{\g}  ,\,\,\,\frac{3\bar{b}-2}{3(\g-1)}\right\} .\enn

By    \eqref{z8}, we see that   $\beta=\frac{\bar{b}}{\g}<\frac{2}{3-\a_{0}}<1$ is always valid for all $\a_{0}\in (0,1).$ In order for \eqref{z8} and \eqref{bb7} to be satisfied,  from \eqref{bb5} we have
\bnn  \g>\frac{3-\a_{0}}{2}\bar{b}>\bar{b}=\frac{4\g+7\th}{3(\g+\th)}>\frac{4}{3}+\frac{\a_{0}}{3}\rightarrow\frac{5}{3}\quad({\rm as}\,\,\a_{0}\rightarrow1).\enn
If   $\beta=\frac{3\bar{b}-2}{3(\g-1)}$ (we have no need  checking \eqref{z8} any more), to guarantee $\beta<1$, it suffices  to require \bnn \g>\frac{1}{3}+\bar{b}=\frac{1}{3}+\frac{4\g+7\th}{3(\g+\th)}> \frac{5}{3}+\frac{\a_{0}}{3}\rightarrow\frac{5}{3}\quad({\rm as}\,\,\a_{0}\rightarrow0).\enn
 The proof of Lemma \ref{lem5.2}  is completed.  \hfill $\Box$\\

By  Lemma \ref{lem5.2},  we can take the following limits, subject to some subsequence,
\be\la{6b10} (\na u_{\de},\,\na \mu_{\de})\rightharpoonup  (\na u,\,\na \mu)\,\,{\rm in}\,\, L^{2},  \ee
\be\la{6b11a} (u_{\de},\, \mu_{\de}) \rightarrow    (u,\,\mu)\,\,\,\,{\rm in}\quad L^{p_{1}}\,\,\,(1\le p_{1}<6),\ee
\be\la{6b11b}  c_{\de} \rightarrow  c\,\,{\rm in}\,\,\,\, W^{1,p_{2}}\,\,\,\,({\rm for\,some}\,p_{2}>2),\ee
\be\la{6b15}   \de \n_{\de}^{4+\th} \rightarrow 0\,\,\,\, {\rm in}\,\,\,\, L^{1}, \quad {\rm and}\quad \n_{\de}  \rightharpoonup \n \,\,{\rm in}\,\, L^{\g+\th},\ee
where   \eqref{6b15}  is due to
$ \n_{\de}^{\g+\th}\le (\g-1)\n_{\de}^{2+\th} \frac{\p f}{\p \n_{\de}}$.
As a result of  \eqref{6b11a}-\eqref{6b15},
\be\la{6b14}  (\n_{\de}  u_{\de},\,\n_{\de}  \mu_{\de})\rightharpoonup (\n u,\,\n\mu)\,\,\,  {\rm in}\,\,\,  L^{p_{3}}\,\,\, ({\rm for\,some}\, p_{3}>6/5),\ee
\be\la{6b14a}  (\n_{\de}  u_{\de}\otimes u_{\de},\,\n_{\de} u_{\de}c_{\de})\rightharpoonup (\n u\otimes u,\,\n u c)\,\,\, {\rm in}\quad  L^{p}\,\,({\rm for\,some}\,p>1);\ee
and furthermore,
\be\la{6b12a}\ba   \n_{\de}^{2} \frac{\p f}{\p \n_{\de}}&=(\g-1)\n_{\de}^{\g}+\n_{\de}H_{1}(c_{\de})\rightharpoonup (\g-1)\overline{\n^{\g}}+\n H_{1}(c):=\overline{\n^{2} \frac{\p f}{\p \n}}\quad {\rm in}\quad L^{\frac{\g+\th}{\g}},\ea\ee
\be\la{6b18} \ba  \n_{\de}  \frac{\p f}{\p c_{\de}}
&=\n_{\de}\ln\n_{\de}H_{1}'(c_{\de})+\n_{\de}H_{2}'(c_{\de}) \rightharpoonup \overline{\n\ln \n}H_{1}'(c)+\n H_{2}'(c):=\overline{ \n  \frac{\p f}{\p c}}\,\, {\rm in}\,\,\,\, L^{p_{3}}.\ea\ee
With  \eqref{6b10}-\eqref{6b18} in hand,  we are able to take $\de$-limit  in \eqref{n7}  and obtain the  following equations in the distribution sense:
\begin{equation}\label{n7b}
\left\{\ba
&{\rm div}(\n u) =0,\\
&{\rm div} (\n u\otimes u)+\na   \left(\overline{\n^{2}  \frac{\p f}{\p \n}} \right)
 =\div \left(\mathbb{S}_{ns}+\mathbb{S}_{c} \right)+\n g_{1}+g_{2},\\
&\div (\n u c)=\lap \mu,\\
&\n\mu=\overline{\n  \frac{\p f}{\p c}}-\lap c.\ea \right.
\end{equation}
In order to complete the proof of   Theorem \ref{t},
it remains  to  verify
\bnn  \overline{\n^{2} \frac{\p f}{\p \n}}=\n^{2} \frac{\p f}{\p \n}\quad {\rm and}\quad \overline{\n  \frac{\p f}{\p c}}=\n  \frac{\p f}{\p c}.\enn
To  this end    it suffices to prove  $\n_{\de}\rightarrow \n \,\,\, {\rm in}\,\,\,L^{1},$ which is our   task in the rest of the paper.

Let   $T_{k}(z)$ be  an increasing and concave  function, in particular,  \be\la{7.0}
 C^{1}([0,\infty))\ni T_{k}(z)=\left\{\ba &z,\quad z\le k\in \mathbb{N},\\
 &k+1,\quad z\ge k+1.\ea \right.\ee
Clearly,
 \be\la{9.1} T_{k}(\n_{\de})\rightharpoonup \overline{T_{k}(\n)}\quad {\rm in}\quad L^{p}(\om),\,\,\,\forall\,\,\,p\in [1,\infty].\ee

 %The next lemma   corresponds  to Lemma \ref{lem4.3} in previous Section.
 \begin{lemma}\la{lem5.4} Let  $(\n_{\de},u_{\de},\mu_{\de},c_{\de})$ be a solution   obtained   in Theorem \ref{t5.1}.  Then, for the effective viscous flux the following holds,
\be\la{7.1} \ba
& \lim_{\de\rightarrow0}\int T_{k}(\n_{\de})\left( \n_{\de}^{2} \frac{\p f}{\p \n_{\de}}-(2\lambda_{1}+\lambda_{2})  \div u_{\de}\right) =
\int \overline{T_{k}(\n)}\left( \overline{\n^{2} \frac{\p f}{\p \n}}-(2\lambda_{1}+\lambda_{2})\div u \right),\ea\ee
where $T_{k}$ is defined in \eqref{7.0}.
 \end{lemma}

\begin{proof}  {The argument is similar to  that in Lemma \ref{lem4.3}.}
 Choose  $\Phi=\phi \na\lap^{-1}(T_{k}(\n_{\de}))$ in   $\eqref{e9}$     to  get
\bnn\ba &\int \phi T_{k}(\n_{\de})\left(\n_{\de}^{2} \frac{\p f}{\p \n_{\de}} -(2\lambda_{1}+\lambda_{2})\div u_{\de}\right)\\
&= -\int\de \phi T_{k}(\n_{\de})\n_{\de}^{4}+\p_{i}\lap^{-1}(T_{k}(\n_{\de})) \p_{i}\phi \left(\de\n_{\de}^{4}+\n_{\de}^{2} \frac{\p f}{\p\n_{\de}} -(\lambda_{1}+\lambda_{2})\div u_{\de}\right) \\
&\quad +\lambda_{1} \int \p_{j}u^{i}_{\de}\p_{i}\lap^{-1}(T_{k}(\n_{\de})) \p_{j}\phi- u^{i}_{\de}\p_{j}\p_{i}\lap^{-1}(T_{k}(\n_{\de}))\p_{j}\phi+  T_{k}(\n_{\de}) u_{\de} \cdot\na \phi\\
&\quad-\int(\n_{\de}g_{1}+g_{2})\phi  \p_{i}\lap^{-1}(T_{k}(\n_{\de}))\\
&\quad+\frac{1}{2}\int|\na c_{\de}|^{2}\left(\phi T_{k}(\n_{\de})+\p_{i}\phi\p_{i}\lap^{-1}(T_{k}(\n_{\de}))\right)\\
&\quad -\int \na c_{\de}\otimes \na c_{\de}\left(\phi\p_{j}\p_{i}\lap^{-1}(T_{k}(\n_{\de}))+\p_{j}\phi\p_{i}\lap^{-1}(T_{k}(\n_{\de}))\right)\\
&\quad-\int \n_{\de}u_{\de}^{j}u_{\de}^{i}\p_{j}\phi \p_{i}\lap^{-1}(T_{k}(\n_{\de}))
\\&\quad - \int u_{\de}^{i}\phi\left[\n_{\de}u_{\de}^{j}\phi \p_{j}\p_{i}\lap^{-1}(T_{k}(\n_{\de}))-T_{k}(\n_{\de})\p_{i}\p_{j}\lap^{-1}(\n_{\de}u^{j}_{\de})\right]\\
&= \sum_{i=1}^{7}R_{i}^{\de}.
\ea\enn
On the other hand,  if we  use  $\phi\na \lap^{-1}\left(\overline{T_{k}(\n)}\right)$ as a test function in  $\eqref{n7b}_{2}$, we infer
 \bnn\ba &\int  \phi  \overline{T_{k}(\n)} \left( \overline{\n^{2} \frac{\p f}{\p \n}}-(2\lambda_{1}+\lambda_{2})\div u \right)\\
&= -\int \p_{i}\lap^{-1}\left(\overline{T_{k}(\n)}\right) \p_{i}\phi \left( \overline{\n^{2} \frac{\p f}{\p \n}}-(\lambda_{1}+\lambda_{2})\div u \right) \\
& \quad+\lambda_{1}\int  \p_{j}u^{i} \p_{i}\lap^{-1}\left(\overline{T_{k}(\n)}\right)\p_{j}\phi-  u^{i} \p_{j}\p_{i}\lap^{-1}\left(\overline{T_{k}(\n)}\right)\p_{j}\phi+  \overline{T_{k}(\n)} u  \cdot\na \phi \\
& \quad-\int (\n g_{1}+g_{2})\phi  \p_{i}\lap^{-1}\left(\overline{T_{k}(\n)}\right)\\
&\quad+\frac{1}{2}\int |\na c |^{2}\left(\phi  \overline{T_{k}(\n)} +\p_{i}\phi\p_{i}\lap^{-1}\left(\overline{T_{k}(\n)}\right)\right) \\
& \quad-\int \na c \otimes \na c \left(\phi\p_{j}\p_{i}\lap^{-1}\left(\overline{T_{k}(\n)}\right)+\p_{j}\phi\p_{i}\lap^{-1}\left(\overline{T_{k}(\n)}\right)\right) \\
& \quad-\int \n u^{j}u^{i}\p_{j}\phi \p_{i}\lap^{-1}\left(\overline{T_{k}(\n)}\right)\\&\quad - \int u^{i}\phi\left[\n u^{j}\p_{j}\p_{i}\lap^{-1}\left(\overline{T_{k}(\n)}\right)-\overline{T_{k}(\n)}\p_{i}\p_{j}\lap^{-1}(\n u^{j})\right]\\
&=\sum_{i=1}^{7}R_{i}.
\ea\enn
Therefore,  we obtain  \eqref{7.1}  provided
\be\la{oo}\lim_{\de\rightarrow0}R_{i}^{\de}=R_{i}\quad (i=1, 2, \cdots, 7).\ee
In fact,    \eqref{oo}  can be verified  by  modifying  slightly   the  argument   in Lemma \ref{lem4.3}.  The detail  is omitted  here.
\end{proof}

Finally, let us   prove  the strong convergence of density.
 By  \eqref{b1}  and the simple fact
\bnn  (\n_{\de}^{\g}-\n^{\g})\left(T_{k}(\n_{\de})-T_{k}(\n)\right)\ge \left(T_{k}(\n_{\de})-T_{k}(\n)\right)^{\g+1},\enn  one has
 \bnn\ba &\int  (\n_{\de}^{2} \frac{\p f(\n_{\de},c_{\de})}{\p\n_{\de}}-\n^{2} \frac{\p f}{\p \n})\left(T_{k}(\n_{\de})-T_{k}(\n)\right)\\
 &=\int  (\g-1) (\n_{\de}^{\g}-\n^{\g})\left(T_{k}(\n_{\de})-T_{k}(\n)\right)\\&\quad+ \int  (\n_{\de}H_{1}(c_{\de})-\n H_{1}(c))\left(T_{k}(\n_{\de})-T_{k}(\n)\right)\\
 &\ge \int (\g-1)\left(T_{k}(\n_{\de})-T_{k}(\n)\right)^{\g+1}  \\
 & \quad+\int \n(H_{1}(c_{\de})- H_{1}(c))\left(T_{k}(\n_{\de})-T_{k}(\n)\right)+ (\n_{\de}-\n)H_{1}(c_{\de})\left(T_{k}(\n_{\de})-T_{k}(\n)\right)\\
  &\ge \int (\g-1)\left(T_{k}(\n_{\de})-T_{k}(\n)\right)^{\g+1} +\int \n(H_{1}(c_{\de})- H_{1}(c))\left(T_{k}(\n_{\de})-T_{k}(\n)\right).\ea\enn  Consequently,
   \be\la{7.13}\ba &\lim_{\de\rightarrow0}\int  (\n_{\de}^{2} \frac{\p f(\n_{\de},c_{\de})}{\p\n_{\de}}-\n^{2} \frac{\p f}{\p \n})\left(T_{k}(\n_{\de})-T_{k}(\n)\right)\\&\ge (\g-1)\lim_{\de\rightarrow0}\int \left(T_{k}(\n_{\de})-T_{k}(\n)\right)^{\g+1}.\ea\ee
By virtue of \eqref{7.13} and \eqref{7.1},  %one has
 \be\ba\la{7.12} &(2\lambda_{1}+\lambda_{2}) \lim_{\de\rightarrow0}\int \left(T_{k}(\n_{\de})\div u_{\de}-\overline{T_{k}(\n)}\div u\right)\\
  &=\lim_{\de\rightarrow0}\int \left( T_{k}(\n_{\de})\n_{\de}^{2} \frac{\p f(\n_{\de},c_{\de})}{\p \n_{\de}}-\overline{T_{k}(\n)}\,\,\overline{\n^{2} \frac{\p f}{\p\n}}\right)\\
  &= \lim_{\de\rightarrow 0}\int  (\n_{\de}^{2} \frac{\p f(\n_{\de},c_{\de})}{\p\n_{\de}}-\n^{2} \frac{\p f}{\p\n})\left(T_{k}(\n_{\de})-T_{k}(\n)\right)\\&\quad+ \int (\overline{\n^{2} \frac{\p f}{\p\n}}-\n^{2} \frac{\p f}{\p\n})\left(T_{k}(\n)-\overline{T_{k}(\n)}\right) \\
  &\ge  \lim_{\de\rightarrow 0}\int (\n_{\de}^{2} \frac{\p f(\n_{\de},c_{\de})}{\p\n_{\de}}-\n^{2} \frac{\p f}{\p \n})\left(T_{k}(\n_{\de})-T_{k}(\n)\right),\ea\ee
  where the last inequality   is due to  the concavity of $T_{k}$ and
  \bnn\ba \overline{\n^{2} \frac{\p f}{\p \n}}&=(\g-1)\overline{\n^{\g}}+\overline{\n\ln\n}H_{1}(c) \ge (\g-1) \n^{\g}+ \n\ln\n H_{1}(c)=\n^{2} \frac{\p f}{\p \n}.\ea\enn

 {Following  \cite{jiang}, we define}
 \bnn L_{k}=\left\{\ba &z\ln z,\quad\quad\quad\quad z\le k,\\
 &z\ln k+z\int_{k}^{z}\frac{T_{k}(s)}{s^{2}}ds,\quad z\ge k.\\
 \ea\right.\enn
A direct computation shows that \bnn b_{k}(z)=L_{k}(z)-\left(\ln k+\int_{k}^{k+1}\frac{T_{k}(s)}{s^{2}}+1\right)z\enn
belongs to   $C([0,\infty))\cap C^{1}((0,\infty)),$    $b_k'(z)=0$ if $z\ge k+1$,  and  $b_{k}'(z)z-b_{k}(z)=T_{k}(z).$
Choosing $b=b_{k}$    in   \eqref{wq},  we infer (approximating  $b_{k}(z)$ near $z=0$)
\bnn  \div (b_{k}(\n) u)+ T_{k}(\n)\div u=0\quad {\rm in}\,\,\,\mathcal{D}(\r),\enn
which implies
 \be\la{8.2} \int T_{k}(\n)\div u=0.\ee
Also,  one has
  \be\la{8.3} \int \overline{T_{k}(\n)\div u} =\lim_{\de\rightarrow 0}\int T_{k}(\n_{\de})\div u_{\de}=0.\ee
From \eqref{8.2}-\eqref{8.3} we  obtain
\be\ba\la{8.4} &C\| T_{k}(\n) -\overline{T_{k}(\n)} \|_{L^{2}}\\
&\ge (2\lambda_{1}+\lambda_{2}) \int \left(T_{k}(\n) -\overline{T_{k}(\n)}\right)\div u\\
&=(2\lambda_{1}+\lambda_{2}) \int \overline{T_{k}(\n)\div u}-\overline{T_{k}(\n)}\div u\\
&=(2\lambda_{1}+\lambda_{2})\lim_{\de\rightarrow0}\int \left(T_{k}(\n_{\de})\div u_{\de}-\overline{T_{k}(\n)}\div u\right)\\
  &\ge  \lim_{\de\rightarrow 0}\int (\n_{\de}^{2}\frac{\p f}{\p \n_{\de}}-\n^{2}\frac{\p f}{\p \n})\left(T_{k}(\n_{\de})-T_{k}(\n)\right)\\
  &\ge (\g-1)\lim_{\de\rightarrow 0}\int \left(T_{k}(\n_{\de})-T_{k}(\n)\right)^{\g+1},\ea\ee
  where the inequalities are due to \eqref{7.13}-\eqref{7.12}.
Therefore, \eqref{8.4} gives
\be\la{bb10}\ba&
\lim_{k\rightarrow\infty} \lim_{\de\rightarrow0}\|T_{k}(\n) - T_{k}(\n_{\de})\|_{L^{\g+1}}^{\g+1}\\
    &\le C\lim_{k\rightarrow\infty}\|T_{k}(\n)-\overline{T_{k}(\n)}\|_{L^{2}} \\
 &\le C \lim_{k\rightarrow\infty}\lim_{\de\rightarrow0}\left(\|T_{k}(\n)-\n\|_{L^{2}}+\|T_{k}(\n_{\de})-\n_{\de}\|_{L^{2}} \right).\ea\ee
However,  by Lemma \ref{lem5.2},     $$\|\n_{\de}\|_{L^{2}}^{2}\le C\|\n_{\de}\|_{L^{\g+\th}}^{\g+\th}\le \|\n_{\de}^{2+\th}\frac{\p f}{\p\n_{\de}}\|_{L^{1}}\le C.$$    Then,  the following estimate holds true
\bnn\ba  \| T_{k}(\n_{\de})-\n_{\de}\|_{L^{p}}
& = \| T_{k}(\n_{\de})-\n_{\de}\|_{L^{2}(\{\n_{\de}\ge k\})} \\&\le 2 \|\n_{\de}\|_{L^{2}(\{\n_{\de}\ge k\})}\le C k^{\frac{p-\g}{2}}\rightarrow0  \quad{\rm as}\,\, k\rightarrow \infty,\ea\enn
which is  uniform in $\de$.  Consequently,  \be\la{ddf} \lim_{k\rightarrow\infty}\|\overline{T_{k}(\n)}-\n\|_{L^{2}}\le  \lim_{k\rightarrow\infty}\lim_{\de\rightarrow0}\| T_{k}(\n_{\de})-\n_{\de}\|_{L^{2}}=0.\ee
The same argument yields
 \be\la{7.6} \ba \lim_{k\rightarrow\infty}\| T_{k}(\n)-\n\|_{L^{2}}=0.\ea\ee
 In terms of \eqref{bb10}-\eqref{7.6}, one has
\bnn \ba&\lim_{k\rightarrow\infty}\lim_{\de\rightarrow0}\|\n_{\de}-\n\|_{L^{1}}\\&
\le \lim_{k\rightarrow\infty}\lim_{\de\rightarrow0}\left(\|\n_{\de}-T_{k}(\n_{\de})\|_{L^{1}}+\|T_{k}(\n_{\de})-T_{k}(\n)\|_{L^{1}}+\|T_{k}(\n)-\n\|_{L^{1}}\right)\\&
=0.\ea\enn

  The proof of Theorem \ref{t} is completed.

   \bigskip

 \section*{Acknowledgement}

 The research of  Z. Liang was partially supported by the fundamental research funds for central universities (JBK 1805001).
 The research  of D. Wang was partially supported by the National Science Foundation under grants DMS-1613213 and DMS-1907519.
 The authors would like to thank the referees for their careful reading of the manuscript and for their valuable comments, corrections, and suggestions.
 
\end{document}